\documentclass[12pt]{amsart}

\usepackage{amsmath}
\usepackage{amsmath,amsfonts,amssymb,amsthm}
\usepackage{graphicx,color}
\DeclareMathRadical{\sqrtsign}{symbols}{"70}{largesymbols}{"70}
\providecommand{\abs}[1]{\lvert#1\rvert}

\newlength{\figboxwidth}             
\newcommand{\grad}{\nabla}
\newcommand{\infinity}{\infty}

\def\@ifundefined#1#2#3%
  {\expandafter\ifx\csname#1\endcsname\relax#2\else#3\fi}

\@ifundefined{theoremstyle}{
}{
\theoremstyle{plain} 
}
\newtheorem{theorem}{Theorem}[section]

\newtheorem{proposition}[theorem]{Proposition}
\newtheorem{lemma}[theorem]{Lemma}

\newtheorem{corollary}[theorem]{Corollary}

\@ifundefined{theoremstyle}{
}{
\theoremstyle{definition} 
}
\newtheorem{definition}[theorem]{Definition}

\newtheorem{remark}[theorem]{Remark}

\mathchardef\GG="321D
\newcommand{\mcc}[1]{{}}

\numberwithin{equation}{section}

\title[Discontinuities in generic dynamical Lagrange spectrum]
{On the discontinuities of Hausdorff dimension in generic dynamical Lagrange spectrum}

\author{Christian Camilo Silva Villamil}

\address{Christian Camilo Silva Villamil: 
IMPA, Instituto de Matem\'atica Pura e Aplicada, Brazil; SUSTech International Center for Mathematics, P. R. China.
}
\email{ccsilvav@impa.br}


\begin{document}

\begin{abstract}
Let $\varphi_0$ be a $C^2$-conservative diffeomorphism of a compact surface $S$ and let $\Lambda_0$ be a mixing horseshoe of $\varphi_0$. Given a smooth real function $f$ defined in $S$ and some diffeomorphism $\varphi$, close to $\varphi_0$, let $\mathcal{L}_{\varphi, f}$ be the Lagrange spectrum associated to the hyperbolic continuation $\Lambda(\varphi)$ of the horseshoe $\Lambda_0$ and $f$. We show that, for generic choices of $\varphi$ and $f$, if $L_{\varphi, f}$ is the map that gives the Hausdorff dimension of the set $\mathcal{L}_{\varphi, f}\cap (-\infty, t)$ for $t\in \mathbb{R}$, then there is an interval $I_{\varphi,f}$ which contains all the discontinuities of $L_{\varphi, f}$ with the end points of this interval the only possible limits of an infinite sequence of discontinuities.
 
\end{abstract}

\maketitle
\tableofcontents
\section{Introduction}\label{s.introduction}
\subsection{Classical spectra}\label{ss.classical-Markov-Lagrange} 
The classical Lagrange and Markov spectra are closed subsets of the real line related to Diophantine approximations. They arise naturally in the study of rational approximations of irrational numbers and of indefinite binary quadratic forms, respectively. More precisely, given an irrational number $\alpha$, let 
$$\ell(\alpha):=\limsup_{\substack{p, q\to\infty \\ p, q\in\mathbb{N}}}\frac{1}{|q(q\alpha-p)|}$$
be its best constant of Diophantine approximation. The set 
$$\mathcal{L}:=\{\ell(\alpha):\alpha\in\mathbb{R}-\mathbb{Q}\ \text{and}\ \ell(\alpha)<\infty\}$$ 
consisting of all finite best constants of Diophantine approximations is the so-called \emph{Lagrange spectrum}. 

Similarly, given a real quadratic form $q(x,y)=ax^2+bxy+cy^2$, let $\Delta(q)=b^2-4ac$ be its discriminant. We define the \emph{Markov spectrum} as follows
$$\mathcal{M}:=\left\{\frac{\sqrt{\Delta(q)}}{\inf\limits_{(x,y)\in\mathbb{Z}^2-\{(0,0)\}} |q(x,y)|} < \infty: q\ \text{is indefinite and}\ \Delta(q)>0\right\}.$$
 
The reader can find more information about the structure of these sets in the classical book \cite{CF} of Cusick and Flahive, but let us mention here that:
\begin{itemize}
\item Markov showed that $\mathcal{L}\cap(-\infty, 3)=\mathcal{M}\cap(-\infty, 3)=\{\sqrt{9-4/z_n^2}:n\in \mathbb{N}\}$ where $z_n$ are the \emph{Markov numbers}, that is, the largest coordinate of a triple $(x_n,y_n,z_n)\in \mathbb{N}^3$ verifying the Markov equation 
$$x_n^2+y_n^2+z_n^2=3x_ny_nz_n.$$
\item Hall showed that $\mathcal{L}$ (and then $\mathcal{M}$) contains a half-line and Freiman determined the biggest half-line contained in the spectra, namely $[c,+\infty)$ where 
$$c=\frac{2221564096+283748\sqrt{462}}{491993569}\simeq 4.52782956\dots$$ 
\item Moreira proved in \cite{M3} several results on the geometry of the Markov and Lagrange spectra, for example, that the map $d:\mathbb{R} \rightarrow [0,1]$, given by
$$d(t)=HD(\mathcal{L}\cap(-\infty,t))= HD(\mathcal{M}\cap(-\infty,t)),$$
(where $HD(X)$ denotes the Hausdorff dimension of the set $X$) is continuous, surjective and such that $\max\{t\in\mathbb{R}:d(t)=0\}=3.$
\end{itemize}

For our purposes, it is worth to point out here that the Lagrange and Markov spectra have the following \emph{dynamical} interpretation in terms of the continued fraction algorithm: Denote by $[a_0,a_1,\dots]$ the continued fraction $a_0+\frac{1}{a_1+\frac{1}{\ddots}}$. Let $\Sigma=\mathbb{N}^{\mathbb{Z}}$ the space of bi-infinite sequences of positive integers, $\sigma:\Sigma\to\Sigma$ be the left-shift map $\sigma((a_n)_{n\in\mathbb{Z}}) = (a_{n+1})_{n\in\mathbb{Z}}$, and let $f:\Sigma\to\mathbb{R}$ be the function
$$f((a_n)_{n\in\mathbb{Z}}) = [a_0, a_1,\dots] + [0, a_{-1}, a_{-2},\dots].$$
Then, 
$$\mathcal{L}=\left\{\limsup_{n\to\infty}f(\sigma^n(\underline{\theta}))<\infty:\underline{\theta}\in\Sigma\right\} \quad \textrm{and} \quad \mathcal{M}= \left\{\sup_{n\to\infty}f(\sigma^n(\underline{\theta}))<\infty:\underline{\theta}\in\Sigma\right\}.$$

In the sequel, we consider the natural generalization of this dynamical version of the classical Lagrange and Markov spectra in the context of horseshoes of smooth diffeomorphisms of compact surfaces, where a horseshoe is a non-empty compact invariant hyperbolic set of saddle type which is transitive, locally maximal, and is not reduced to a periodic orbit.
\subsection{Dynamical spectra}

Let $\varphi:S\rightarrow S$ be a diffeomorphism of a $C^{\infty}$ compact surface $S$ with a mixing horseshoe $\Lambda$ and let $f:S\rightarrow \mathbb{R}$ be a differentiable function. Following the above characterization of the classical spectra, we define the maps $\ell_{\varphi,f}: \Lambda \rightarrow \mathbb{R}$ and $m_{\varphi,f}: \Lambda \rightarrow \mathbb{R}$ given by $\ell_{\varphi,f}(x)=\limsup\limits_{n\to \infty}f(\varphi^n(x))$ and $m_{\varphi,f}(x)=\sup\limits_{n\in\mathbb{Z}}f(\varphi^n(x))$ for $x\in \Lambda$ and call $\ell_{\varphi,f}(x)$ the \textit{Lagrange value} of $x$ associated to $f$ and $\varphi$ and also $m_{\varphi,f}(x)$ the \textit{Markov value} of $x$ associated to $f$ and $\varphi$. The sets
$$\mathcal{L}_{\varphi,f}=\ell_{\varphi,f}(\Lambda)=\{\ell_{\varphi,f}(x):x\in \Lambda\}$$
and
$$\mathcal{M}_{\varphi,f}=m_{\varphi,f}(\Lambda)=\{m_{\varphi,f}(x):x\in \Lambda\}$$
are called \textit{Lagrange Spectrum} of $(\varphi,f)$ and \textit{Markov Spectrum} of $(\varphi,f)$. Here, we omitted the reference to the horseshoe $\Lambda$ because in our context it will always be determined by the diffeomorphism $\varphi.$

In this paper, we are interested in the study of the real function 
\begin{equation}\label{f1}
L_{\varphi,f}(t)=HD(\mathcal{L}_{\varphi,f}\cap (-\infty,t)) .
\end{equation}
The description of this function is closely related to the study of the behavior of the family of sets $\{\Lambda_t\}_{t\in \mathbb{R}}$, where for $t\in \mathbb{R}$
$$\Lambda_t=m_{\varphi, f}^{-1}((\infinity,t]) =\bigcap\limits_{n\in\mathbb{Z}}\varphi^{-n}(f|_{\Lambda}^{-1}((\infinity,t])) = \{x\in\Lambda: \forall n\in \mathbb{Z}, \ f(\varphi^n(x))\leq t\}.$$
In order to do that, we will explore the combinatorial nature of $\varphi|_{\Lambda}$ and its connection with the unstable and stable Cantor sets associated to $\Lambda$. More specifically, fix a Markov partition $\{R_a\}_{a\in \mathcal{A}}$ of $\Lambda$ with sufficiently small diameter consisting of rectangles $R_a \sim I_a^u \times I_a^s$ delimited by compact pieces $I_a^s$, $I_a^u$, of stable and unstable manifolds of certain points of $\Lambda$, see \cite[Theorem 2, page 172]{PT93}. The set $\mathcal{B}\subset \mathcal{A}^{2}$ of admissible transitions consists of pairs $(a,b)$ such that $\varphi(R_a)\cap R_{b}\neq \emptyset$; so, we can define the transition matrix $B$ by
$$b_{ab}=1 \ \ \text{if} \ \  \varphi(R_a)\cap R_b\neq \emptyset \ \ \text{and}  \ b_{ab}=0  \ \text{otherwise, for $(a,b)\in \mathcal{A}^{2}$.}$$

Let $\Sigma_{\mathcal{A}}=\left\{\underline{a}=(a_{n})_{n\in \mathbb{Z}}:a_{n}\in \mathcal{A} \ \text{for all} \ n\in \mathbb{Z}\right\}$ and consider the homeomorphism of $\Sigma_{\mathcal{A}}$, the shift, $\sigma:\Sigma_{\mathcal{A}}\to\Sigma_{\mathcal{A}}$ defined by $\sigma(\underline{a})_{n}=a_{n+1}$. Let $\Sigma_{\mathcal{B}}=\left\{\underline{a}\in \Sigma_{\mathcal{A}}:b_{a_{n}a_{n+1}}=1\right\}$, this set is closed and $\sigma$-invariant subspace of $\Sigma_{\mathcal{A}}$. Still denote by $\sigma$ the restriction of $\sigma$ to $\Sigma_{\mathcal{B}}$, the pair $(\Sigma_{\mathcal{B}},\sigma)$ is a subshift of finite type, see \cite[Chapter 10]{{Shub}}. The dynamics of $\varphi$ on $\Lambda$ is topologically conjugate to the sub-shift $\Sigma_{\mathcal{B}}$, namely, there is a homeomorphism $\Pi: \Lambda \to \Sigma_{\mathcal{B}}$ such that $\varphi\circ \Pi=\Pi\circ \sigma$.

As we will generally deal with sequences, we transfer the function $f$ from $\Lambda$ to a function (still denoted $f$) on $\Sigma_{\mathcal{B}}$. In this way, we set 
$$\Sigma_t=\Pi(\Lambda_t)=\{\theta\in\Sigma_{\mathcal{B}}: \sup\limits_{n\in\mathbb{Z}} f(\sigma^n(\theta))\leq t\}.$$

Recall that the stable and unstable manifolds of $\Lambda$ can be extended to locally invariant $C^{1+\alpha}$ foliations in a neighborhood of $\Lambda$ for some $\alpha>0$. Using these foliations it is possible to define projections $\pi^u_a: R_a\rightarrow I^s_a\times \{i^u_a\}$ and $\pi^s_a: R_a\rightarrow \{i^s_a\} \times I^u_a$ of the rectangles into the connected components $I^s_a\times \{i^u_a\}$ and $\{i^s_a\} \times I^u_a$ of the stable and unstable boundaries of $R_a$, where $i^u_a\in \partial I^u_a$ and $i^s_a\in \partial I^s_a$ are arbitrarily fixed. In this way, we have the unstable and stable Cantor sets
$$K^u:=\bigcup_{a\in \mathcal{A}}\pi^s_a(\Lambda\cap R_a) \ \mbox{and} \ K^s:=\bigcup_{a\in \mathcal{A}}\pi^u_a(\Lambda\cap R_a).$$

In fact $K^u$ and $K^s$ are $C^{1+\alpha}$ dynamically defined, associated to some expanding maps $\psi_s$ and $\psi_u$ defined in the following way: If $y\in R_{a_1}\cap \varphi(R_{a_0})$ we put
$$\psi_s(\pi^u_{a_1}(y))=\pi^u_{a_0}(\varphi^{-1}(y))$$
and if $z\in R_{a_0}\cap \varphi^{-1}(R_{a_1})$ we put
$$\psi_u(\pi^s_{a_0}(z))=\pi^s_{a_1}(\varphi(z)).$$

Moreira's theorem of \cite{M3} was generalized first in \cite{CMM16} in the context of {\it conservative} diffeomorphism with some horseshoe with Hausdorff dimension smaller than $1$, and later the condition on the dimension of the horseshoe was removed in \cite{GCD}. More specifically, the authors proved that for typical choices of the conservative diffeomorphism and of the real function, the intersections of the corresponding dynamical Markov and Lagrange spectra with half-lines $(-\infty,t)$ have the same Hausdorff dimension, and this defines a continuous function of $t$ whose image is $[0,\min \{1,D\}]$, where $D$ is the Hausdorff dimension of the horseshoe.

Our main theorem (cf. Theorem \ref{principal0} below) is quite related to the result of the previous paragraph, but in our case, we consider typical diffeomorphism that are not necessarily conservative. Here, we will identify ``the canonical interval" where $L_{\varphi,f}$ can have a discontinuity, and we will show that far away from the endpoints of this interval, one has finitely many points of discontinuity.

\subsection{Statement of the main theorem}  Let $\varphi_0$ be a smooth conservative diffeomorphism of a surface $S$ possessing a mixing horseshoe $\Lambda_0$. Denote by $\mathcal{U}$ a  $C^{2}$ neighborhood of $\varphi_0$ in the space $\textrm{Diff}^{2}(S)$ of smooth diffeomorphisms of $S$ such that $\Lambda_0$ admits a continuation $\Lambda$ for every $\varphi\in\mathcal{U}$. Using the notations of the previous subsection, our objective is to study the discontinuities of the map $L_{\varphi,f}$ defined by 
$$t\mapsto L_{\varphi,f}(t)=HD(\mathcal{L}_{\varphi, f}\cap (-\infty, t)).$$
In order to do this, we consider the interval $I_{\varphi,f}=[c_{\varphi,f},\tilde{c}_{\varphi,f}]$, where
$$c_{\varphi,f}:=\sup\{t\in\mathbb{R}:L_{\varphi,f}(t)=\min L_{\varphi,f}=0\}$$
and 
$$\tilde{c}_{\varphi,f}:=\inf\{t\in\mathbb{R}:L_{\varphi,f}(t)=\max L_{\varphi,f}=HD(\mathcal{L}_{\varphi,f})\}$$
which is the interval where $L_{\varphi,f}$ can have discontinuities. With this notation, our main result is the following

\begin{theorem}\label{principal0} If $\mathcal{U}\subset\textrm{Diff}^{2}(S)$ is sufficiently small, then there exists a residual subset $\mathcal{U}^{*}\subset \mathcal{U}$ with the property that for every $\varphi\in\mathcal{U}^{*}$ and any $r\geq2$, there exists a $C^r$-residual set $\mathcal{P}_{\varphi,\Lambda}\subset C^r(S,\mathbb{R})$ such that given $f\in \mathcal{P}_{\varphi,\Lambda}$ one has
$$\max L_{\varphi, f}=HD(\mathcal{L}_{\varphi,f})=\min \{1,HD(\Lambda)\}$$ 
and 
$$c_{\varphi,f}=\min \mathcal{L}^{'} _{\varphi, f}=\min \{x: x\ \textrm{is an accumulation point of}\ \mathcal{L}_{\varphi, f}\}.$$
Even more,
\begin{itemize}
    \item If $HD(\Lambda)<1$ then $L_{\varphi, f}$ has finitely many discontinuities in any closed subinterval $I\subset I_{\varphi,f}$ that does not contain $c_{\varphi,f}.$
    \item If $HD(\Lambda)\geq 1$ then $L_{\varphi, f}$ has finitely many discontinuities in any closed subinterval $I\subset I_{\varphi,f}$ that does not contain neither $c_{\varphi,f}$ nor $\tilde{c}_{\varphi,f}.$
\end{itemize}

\end{theorem}

As a consequence, we immediately have the corollaries 
\begin{corollary}
If $HD(\Lambda_0)<1$, then by choosing $\mathcal{U}$ small, given $\varphi\in\mathcal{U}^*$, $f\in \mathcal{P}_{\varphi,\Lambda}$ and $\epsilon >0$ the function $L_{\varphi, f}$ has finitely many discontinuities in the interval $[c_{\varphi,f}+\epsilon, \infinity)$. Therefore, $c_{\varphi,f}$ is the only possible limit of an infinite sequence of discontinuities of $L_{\varphi, f}$.
\end{corollary}
\begin{corollary}
If $HD(\Lambda_0)>1$, then by choosing $\mathcal{U}$ small, given $\varphi\in\mathcal{U}^*$, $f\in \mathcal{P}_{\varphi,\Lambda}$ and $\epsilon >0$ small, the function $L_{\varphi, f}$ has finitely many discontinuities in the interval $[c_{\varphi,f}+\epsilon, \tilde{c}_{\varphi,f}-\epsilon].$ Therefore, $c_{\varphi,f}$ and $\tilde{c}_{\varphi,f}$  are the only possible limits of an infinite sequence of discontinuities of $L_{\varphi, f}$.
\end{corollary}

\section{Preliminary results} \label{pre}
\subsection{Stable and unstable dimensions}
Given a Markov partition $\mathcal{P}=\{R_a\}_{a\in \mathcal{A}}$, recall that the geometrical description of $\Lambda$ in terms of the Markov partition $\mathcal{P}$ has a combinatorial counterpart in terms of the Markov shift $\Sigma_{\mathcal{B}}\subset \mathcal{A}^{\mathbb{Z}}$.
Given an admissible finite sequence $\alpha=(a_1,...,a_n)\in \mathcal{A}^n$ (i.e., $(a_i,a_{i+1})\in \mathcal{B}$ for all $1\le i<n$), we define
	$$I^u(\alpha)=\{x\in K^u: \psi_u^i(x)\in I^u(a_i,a_{i+1}),\ i=1,2,...,n-1\}$$
and if $\alpha^T=(a_n,a_{n-1},...,a_1),$ we define
	$$I^s(\alpha^T)=\{y\in K^s: \psi_s^i(y)\in I^s(a_{i},a_{i-1}),\ i=2,...,n\}.$$
In a similar way, let $\theta=(a_{s_1},a_{s_1+1},...,a_{s_2})\in \mathcal{A}^{s_2-s_1+1}$ an admissible word where $s_1, s_2 \in \mathbb{Z}$, $s_1 < s_2$ and fix $s_1\le s\le s_2$. Define $$R(\theta;s)=\bigcap_{m=s_1-s}^{s_2-s} \varphi^{-m}(R_{a_{m+s}}).$$
Note that if $x\in R(\theta;s)\cap \Lambda$, then the symbolic representation of $x$ is of the form $\Pi(x)=(\dots,a_{s_1}\dots a_{s-1};a_{s},a_{s+1}\dots a_{s_2}\dots)$, where the letter following to $;$ is in the $0$ position of the sequence.

In our context of dynamically defined Cantor sets, we can relate the length of the unstable and stable intervals determined by an admissible word to its length as a word in the alphabet $\mathcal{A}$ via the {\it bounded distortion property} that lets us conclude that for some constant $c_1>0$
\begin{equation}\label{bdp1}
	e^{-c_1}\le \dfrac{|I^u(\alpha\beta)|}{|I^u(\alpha)|\cdot|I^u(\beta)|}\le e^{c_1} \ \mbox{and} \ e^{-c_1}\le \dfrac{|I^s((\alpha\beta)^T)|}{|I^s(\alpha^T)|\cdot|I^s(\beta^T)|}\le e^{c_1},
	\end{equation}
and also, for some positive constants $\lambda_1,\lambda_2<1$, one has
\begin{equation}\label{bdp2}
	e^{-c_1} \lambda_1^{\abs{\alpha}}\leq\abs{I^u(\alpha)}\leq e^{c_1} \lambda_2^{\abs{\alpha}} \ \mbox{and} \ e^{-c_1} \lambda_1^{\abs{\alpha}}\leq\abs{I^s(\alpha^T)}\leq e^{c_1} \lambda_2^{\abs{\alpha}}.
	\end{equation}
We write $r^{(u)}(\alpha)$ for the unstable scale of $\alpha$, that is, $r^{(u)}(\alpha)=\lfloor \log(1/\abs{I^u(\alpha)})\rfloor$ and similarly, $r^{(s)}(\alpha) =\lfloor \log(1/\abs{I^s(\alpha^T)})\rfloor$ for the stable scale of $\alpha$. 
Write $\alpha^{\ast}=(a_1,a_2,...,a_{n-1})$ if $\alpha=(a_1,a_2,...,a_n)$ and for $r\in \mathbb{N}$ define the sets
$$P^{(u)}_r=\{\alpha\in \mathcal{A}^n \ \mbox{admissible}:  r^{(u)}(\alpha)\geq r \ \mbox{and} \ r^{(u)}(\alpha^{\ast})<r\}$$
and
$$P^{(s)}_r=\{\alpha\in \mathcal{A}^n \ \mbox{admissible}:  r^{(s)}(\alpha)\ge r \ \mbox{and} \ r^{(s)}(\alpha^{\ast})<r\}.$$

Now, given any $X\subset \Lambda$ compact and $\varphi$-invariant we define its projections 
$$\pi^u(X)=\bigcup_{a\in \mathcal{A}} \pi^s_a(X\cap R_a) \ \mbox{and} \ \pi^s(X)=\bigcup_{a\in \mathcal{A}}\pi^u_a(X\cap R_a).$$
We also set
$$\mathcal{C}_u(X,r)=\{\alpha\in P^{(u)}_r:  I^{u}(\alpha)\cap \pi^u(X)\neq \emptyset\}$$
and
$$ \mathcal{C}_s(X,r)=\{\alpha\in P^{(s)}_r:  I^{s}(\alpha^T)\cap \pi^s(X)\neq \emptyset\}$$
whose cardinalities are denoted $N_u(X,r)=|\mathcal{C}_u(X,r)|$ and $N_s(X,r)=|\mathcal{C}_s(X,r)|$. 

Note that by (\ref{bdp2}) for $\alpha \in \mathcal{C}_u(X,r)$ one has $e^{c_1}\lambda_2^{-1}\lambda_2^{\abs{\alpha}}>\abs{I^u(\alpha^{\ast})} > e^{-r}$
and it follows from this that $\abs{\alpha}<r/\log(\lambda_2^{-1})+\log (e^{c_1}\lambda_2^{-1})/\log (\lambda_2^{-1})$ and then 
\begin{equation} \label{expN}
  N_u(X,r)=\abs{\mathcal{C}_u(X,r)} \leq e^{\alpha_1r+\alpha_2}
\end{equation}
where $\alpha_1=\log\abs{\mathcal{A}}/\log(\lambda_2^{-1}) >0$ and $\alpha_2=\log (e^{c_1}\lambda_2^{-1})\cdot \log\abs{\mathcal{A}} /\log (\lambda_2^{-1})>0$ depend only on $\varphi$ and $\Lambda$. This computations also give for $\alpha \in \mathcal{C}_u(X,r)$ that 
\begin{equation}\label{words}
    \abs{\alpha}<\alpha_1r+\alpha_2.\end{equation}
Note that the same inequalities hold for $N_s(X,r)$ and words in $\mathcal{C}_s(X,r)$.

In the article \cite{CMM16} the authors proved the following lemma in the case of $X=\Lambda_{t}$, where $t\in \mathbb{R}$. For completeness, we give a proof here:

\begin{lemma}\label{c_2}
There exists a constant $c_2= c_2(\varphi, \Lambda)\in \mathbb{N}$ such that
if $X$ is a compact, $\varphi$-invariant subset of $\Lambda$, then
	$$N_u(X,m+n)\le |\mathcal{A}|^{c_2}\cdot N_u(X,m)\cdot N_u(X,n)$$
and
	$$N_s(X,m+n)\le |\mathcal{A}|^{c_2}\cdot N_s(X,m)\cdot N_s(X,n)$$
for all $n, m\in\mathbb{N}$.
\end{lemma}

\begin{proof}By symmetry, it is sufficient to show that the sequence $\{N_u(X,r)\}_{r\in\mathbb{N}}$ satisfies the conclusions of the lemma. By (\ref{bdp1}) and (\ref{bdp2}) we have for all $\alpha$, $\beta$, $\gamma$ finite words such that the concatenation $\alpha\beta \gamma$ is admissible
$$|I^u(\alpha\beta \gamma)|\leq e^{2c_1}|I^u(\alpha)|\cdot |I^u(\beta)|\cdot |I^u(\gamma)|\leq  e^{3c_1}\lambda_2^{\abs{\gamma}}\cdot|I^u(\alpha)|\cdot |I^u(\beta)|.$$
Now, we note that, for each $c\in\mathbb{N}$, one can cover $\pi^u(X)$ with no more than $\abs{\mathcal{A}}^{c}\cdot N_u(X,n)\cdot N_u(X,m)$ intervals $I^u(\alpha\beta \gamma)$ with $\alpha\in \mathcal{C}_u(X,n)$, $\beta\in \mathcal{C}_u(X,m)$, $\gamma\in\mathcal{A}^c$ and $\alpha\beta \gamma$ admissible.

Therefore, by taking $c_2=\left\lceil\frac{3c_1}{\log\lambda_2^{-1}}\right\rceil\in\mathbb{N}$
it follows that we can cover $\pi^u(X)$ with no more than $\abs{\mathcal{A}}^{c_2}\cdot N_u(X,n)\cdot N_u(X,m)$ intervals $I^u(\alpha\beta \gamma)$ whose unstable scales satisfy 
$$r^{(u)}(\alpha\beta \gamma)\geq r^{(u)}(\alpha)+r^{(u)}(\beta)\geq n+m.$$
Hence, by definition, we conclude that 
$$N_\textbf{u}(X,n+m)\leq \abs{\mathcal{A}}^{c_2}\cdot N_u(X,n)\cdot N_u(X,m),$$
as we wanted to see.
\end{proof}

From this lemma we get that for each $X\subset \Lambda$ compact, $\varphi$-invariant there exist the limits 
\begin{equation}\label{du}
    D_u(X)=\lim_{r\to \infty}\dfrac{\log N_u(X,r)}{r}=\inf\limits_{r\in\mathbb{N}}\dfrac{\log (|\mathcal{A}|^{c_2}\cdot N_u(X,r))}{r}
    \end{equation}
and	
\begin{equation}
D_s(X)=\lim_{r\to \infty}\dfrac{\log N_s(X,r)}{r}=\inf\limits_{r\in\mathbb{N}}\dfrac{\log (|\mathcal{A}|^{c_2}\cdot N_s(X,r))}{r}.
\end{equation}
It can also be easily verified that the numbers $D_u(X)$ and $D_s(X)$ are the limit capacities of $\pi^u(X)$ and $\pi^s(X)$ respectively.

By (\ref{bdp2}) we have for the constants $\tilde{C}=\log \lambda_1/\log \lambda_2>1$ and $C=e^{c_1\cdot(\tilde{C}+1)}>1$ and any $\alpha$ admissible that
\begin{equation}\label{conservative}
 C^{-1}\abs{I^u(\alpha)}^{\tilde{C}} \leq \abs{I^s(\alpha^T)} \leq C \abs{I^u(\alpha)}^{1/\tilde{C}}
\end{equation}
and for this, we can conclude again that for every $X\subset \Lambda$, compact and $\varphi$-invariant, $D_s(X)$ and $D_u(X)$ are comparable:
\begin{equation}\label{beta}
 \tilde{C}^{-1} D_u(X) \leq D_s(X) \leq \tilde{C} D_u(X)
\end{equation}
and so,
\begin{equation}\label{Ds}
  HD(X)\leq D_s(X)+D_u(X)\leq (\tilde{C}+1)D_s(X)  
\end{equation}
and 
\begin{equation}\label{Du}
   HD(X)\leq D_s(X)+D_u(X)\leq (\tilde{C}+1)D_u(X). 
\end{equation}

\subsection{Sets of finite type and connection of subhorseshoes}

The following definitions and results can be found in \cite{GC}. Fix a horseshoe $\Lambda$ of some diffeomorphism $\varphi:S\rightarrow S$ and $\mathcal{P}=\{R_a\}_{a\in \mathcal{A}}$ some Markov partition for $\Lambda$. Take a finite collection $X$ of finite admissible words of the form $\theta=(a_{-n(\theta)},\dots,a_{-1},a_0,a_1,\dots,a_{n(\theta)})$, we say that the maximal invariant set 
$$M(X)=\bigcap \limits_{m \in \mathbb{Z}} \varphi ^{-m}(\bigcup \limits_{\theta \in X}  R(\theta;0))$$ 
is a \textit{hyperbolic set of finite type}. Even more, it is said to be a \textit{subhorseshoe} of $\Lambda$ if it is nonempty and $\varphi|_{M(X)}$ is transitive. Observe that a subhorseshoe need not be a horseshoe; indeed, it could be a periodic orbit, in which case it will be called trivial.

By definition, hyperbolic sets of finite type have local product structure. In fact, any hyperbolic set of finite type is a locally maximal invariant set of a neighborhood of a finite number of elements of some Markov partition of $\Lambda$.

\begin{definition}
Any $\tau \subset M(X)$ for which there are two different subhorseshoes $\Lambda(1)$ and $\Lambda(2)$ of $\Lambda$ contained in $M(X)$ with 
$$\tau=\{x\in M(X):\ \omega(x)\subset \Lambda(1)\ \text{and}\ \alpha(x)\subset \Lambda(2)  \}$$
will be called a transient set or transient component of $M(X)$.
\end{definition}

Note that by the local product structure, given a transient set $\tau$ as before,
\begin{equation}
    HD(\tau)=HD(K^s(\Lambda(2)))+HD(K^u(\Lambda(1))).
\end{equation}

\begin{proposition}\label{appendix}
Any hyperbolic set of finite type $M(X)$, associated with a finite collection of finite admissible words $X$ as before, can be written as
$$M(X)=\bigcup \limits_{i\in \mathcal{I}} \tilde{\Lambda}_i $$ 
where $\mathcal{I}$ is a finite index set (that may be empty) and for $i\in \mathcal{I}$,\ $\tilde{\Lambda}_i$ is a subhorseshoe or a transient set.
\end{proposition}

Now, fix $r\geq 2$ and for $x\in \Lambda$, let $e^s_x$ and $e^u_x$ be unit vectors in the stable and unstable directions of $T_xS$. Given some subhorseshoe $\tilde{\Lambda}\subset\Lambda$ we define
$$\mathcal{R}_{\varphi, \tilde{\Lambda}}:=\{f\in C^r(S,\mathbb{R}): \grad f(x) \textrm{ is not perpendicular either to } e^s_x \textrm{  or } e^u_x \textrm{ for all } x\in\tilde{\Lambda}\}.$$ 
In other terms, $\mathcal{R}_{\varphi, \tilde{\Lambda}}$ is the class of $C^r$-functions $f:S\to\mathbb{R}$ that are locally monotone along stable and unstable directions for points in $\tilde{\Lambda}$. The next proposition follows from the results proved in \cite[Remark 1.4]{CMM16}:
\begin{proposition}\label{R-generic} Fix $r\geq 2$. If the subhorseshoe $\tilde{\Lambda}\subset\Lambda$ has Hausdorff dimension smaller than $1$, then $\mathcal{R}_{\varphi, \tilde{\Lambda}}$ is $C^r$-open and dense and for $f\in\mathcal{R}_{\varphi, \tilde{\Lambda}}$ the functions $t\mapsto D_u(\tilde{\Lambda}_{t})$ and $t\mapsto D_s(\tilde{\Lambda}_{t})$ are continuous, where $\tilde{\Lambda}_t = \{x\in\tilde{\Lambda}: \forall n\in \mathbb{Z}, \ f(\varphi^n(x))\leq t\}.$
\end{proposition}

Fix $f:S\rightarrow \mathbb{R}$  differentiable. A notion that plays an important role in our study of the discontinuities of the map $L_{\varphi, f}$ is the notion of \textit{connection of subhorseshoes} 

\begin{definition}\label{conection of horseshoes1}
Given $\Lambda(1)$ and $\Lambda(2)$ subhorseshoes of $\Lambda$ and $t\in \mathbb{R}$, we said that $\Lambda(1)$ \emph{connects} with $\Lambda(2)$ or that $\Lambda(1)$ and $\Lambda(2)$ \emph{connect} before $t$ if there exist a subhorseshoe $\tilde{\Lambda}\subset \Lambda$ and some $q< t$ with $\Lambda(1) \cup \Lambda(2) \subset \tilde{\Lambda}\subset \Lambda_q$.
\end{definition}

For our present purposes, the next criterion of connection will also be important

\begin{proposition}\label{connection11}
Suppose $\Lambda(1)$ and $\Lambda(2)$ are subhorseshoes of $\Lambda$ and for some $x,y \in \Lambda$ we have $x\in W^u(\Lambda(1))\cap W^s(\Lambda(2))$ and $y\in W^u(\Lambda(2))\cap W^s(\Lambda(1))$. If for some $t\in \mathbb{R}$, it is true that 
$$\Lambda(1) \cup \Lambda(2) \cup \mathcal{O}(x) \cup \mathcal{O}(y) \subset \Lambda_t,$$ then for every $\epsilon >0$,\ $\Lambda(1)$ and $\Lambda(2)$ connect before $t+\epsilon$. 
\end{proposition}

 \begin{corollary}\label{connection3}
 Let $\Lambda(1)$,\ $\Lambda(2)$ and $\Lambda(3)$ subhorseshoes of $\Lambda$ and $t\in \mathbb{R}$. If $\Lambda(1)$\ connects with $\Lambda(2)$ before $t$ and $\Lambda(2)$\ connects with $\Lambda(3)$ before $t$. Then also $\Lambda(1)$\ connects with $\Lambda(3)$ before $t$.
 \end{corollary}
 
\section{Proof of Theorem \ref{principal0}}

The proof when the Hausdorff dimension of the horseshoe is less than $1$, is by contradiction and includes several technical steps: we suppose the existence of an infinite sequence of discontinuities $\{t_n\}_{n\in\mathbb{N}}$ of the map $L_{\varphi, f}$ in some closed subinterval of $I_{\varphi,f}$ that does not contain the first accumulation point of the Lagrange spectrum. Then, we associate to each $n$ a pair of subhorseshoes 
$\Lambda^s_n$ and $\Lambda^u_n$ that do not connect before $t_n$ but they connect little time after it. We show that for some $\theta \in \{s,u\}$ the sequence $\{\Lambda^{\theta}_n\}_{n\in\mathbb{N}}$ has the property that given $N\in \mathbb{N}$ arbitrary, there are $n_1<n_2<...<n_N$ such that for $i,j\in  \{1,...,N \}$ with $i\neq j$,\ $\Lambda^{\theta}_{n_i}$ and $\Lambda^{\theta}_{n_j}$ do not connect before $\max \{ t_{n_i}, t_{n_j} \}.$ 

On the other hand, by choosing correct scales (at the level of sequences), we show that for every $n$, we can associate a periodic point $p_n$ (with period bounded by a fixed constant) in such a way that it is possible to connect $\Lambda^{\theta}_n$ and $\Lambda^{\theta}_m$ before $\max \{ t_m, t_n \}$ if $p_n=p_m$. This allows us to obtain the desired contradiction. The proof when the Hausdorff dimension of the horseshoe is greater than or equal to $1$ is reduced to the previous case.

\subsection{The residual subsets} 
In this short subsection, we introduce the residual sets with which we are going to work. First, using the spectral decomposition theorem, it follows the next result from \cite{M50}: 

\begin{proposition}\label{dimension}
There exists a residual subset $\mathcal{U}^*\subset \mathcal{U}$ with the property that for every subhorseshoe $\tilde{\Lambda}\subset \Lambda$ and any $f\in C^1(S,\mathbb{R})$ such that there exists some point in $\tilde{\Lambda}$ with its gradient not parallel to either the stable direction or the unstable direction, one has 
	$$HD(f(\widetilde{\Lambda})) = \min \{1,HD(\widetilde{\Lambda})\}.$$
\end{proposition} 
that we use to prove the next proposition

\begin{proposition} \label{lagrange1}
If $\mathcal{U}^*$ is as in Proposition \ref{dimension} and $r \geq 2$, then for any $\varphi \in \mathcal{U}^*$, there exists a $C^r$-residual subset $\mathcal{P}_{\varphi,\Lambda}$ such that for every subhorseshoe $\widetilde{\Lambda}\subset\Lambda$ and any $f\in \mathcal{P}_{\varphi,\Lambda}$ one has 
$$\min \{1,HD(\widetilde{\Lambda})\}=HD(\ell_{\varphi,f}(\widetilde{\Lambda}))=HD(m_{\varphi,f}(\widetilde{\Lambda})).$$
Even more, if $HD(\tilde{\Lambda})<1$ one has $\mathcal{P}_{\varphi,\Lambda}\subset \mathcal{R}_{\varphi,\tilde{\Lambda}}.$ 
\end{proposition}
\begin{proof}
Following the ideas of the proof of \cite[Theorem 1]{MR2} we see that given a subhorseshoe $\widetilde{\Lambda}\subset\Lambda$, the set 
$$H_{\widetilde{\Lambda}}=\{f\in C^r(S,\mathbb{R}): \abs{M_{\widetilde{\Lambda}, f}}=1\ \text{and if}\ z\in M_{\widetilde{\Lambda}, f},\ Df_z(e^{s,u}_z)\neq 0  \}$$
is $C^r$- open and dense, where $M_{\widetilde{\Lambda}, f}= \{z\in \widetilde{\Lambda}: f(z)=\max f|_{\tilde{\Lambda}} \}$.

If $HD(\tilde{\Lambda})<1$ set $\mathcal{H}_{\widetilde{\Lambda}}=H_{\widetilde{\Lambda}}\cap \mathcal{R}_{\varphi,\tilde{\Lambda}}$ (which is residual by Proposition \ref{R-generic}) and $\mathcal{H}_{\widetilde{\Lambda}}=H_{\widetilde{\Lambda}}$ in other case. Define then  
$$\mathcal{P}_{\varphi,\Lambda}:= \bigcap \limits_{ \substack{\widetilde{\Lambda}\subset\Lambda\  \\ subhorseshoe}}\mathcal{H}_{\widetilde{\Lambda}}.$$

In \cite{MR2} is also proved that for any such subhorseshoe $\widetilde{\Lambda}\subset\Lambda$ and $f \in \mathcal{P}_{\varphi,\Lambda}$ if $x_M$ is the unique element where $f|_{\widetilde{\Lambda}}$ attains its maximum value, then for any $\epsilon>0$ there exists some subhorseshoe $\widetilde{\Lambda}^{\epsilon}\subset \widetilde{\Lambda} \setminus \{ x_M\}$ with 
$$HD(\widetilde{\Lambda}^{\epsilon})\geq HD(\widetilde{\Lambda})(1-\epsilon)$$ 
and such that for some point $d\in \widetilde{\Lambda}^{\epsilon}$ there exists a local $C^{1}$-diffeomorphism $\tilde{A}$ defined in a neighborhood $U_{d}$ of $d$ such that 
$$f(\varphi^{j_0}(\tilde{A}(\tilde{\Lambda}_{j_0})))\subset \ell_{\varphi,f}(\widetilde{\Lambda}),$$
where $j_{0}$ is an integer and $\tilde{\Lambda}_{j_0}\subset \widetilde{\Lambda}^{\epsilon}$ has nonempty interior in $\widetilde{\Lambda}^{\epsilon}$ and then is such that $HD(\tilde{\Lambda}_{j_0})=HD(\widetilde{\Lambda}^{\epsilon})$. Moreover, it is also proved that $\dfrac{\partial \tilde{A}}{\partial e_{x}^{s,u}}\parallel e^{s,u}_{\tilde{A}(x)}$, for $x\in U_{d}\cap  \widetilde{\Lambda}^{\epsilon}$ and then, $\nabla (f\circ \varphi^{j_{0}} \circ \tilde{A})(x)\nparallel e_{x}^{s,u}$ for every $x \in \tilde{\Lambda}_{j_0}$.

Extending properly $f\circ \varphi^{j_{0}} \circ \tilde{A}$, and letting $\epsilon$ tend to $0$; it follows from this and Proposition \ref{dimension} that
$$ \min\{1,HD(\tilde{\Lambda})\}\leq HD(\ell_{\varphi,f}(\widetilde{\Lambda})).$$
An elementary compactness argument shows that  $\{\ell_{\varphi, f}(x):x\in X\}\subset\{m_{\varphi, f}(x): x\in X\}\subset f(X)$
whenever $X\subset M$ is a compact $\varphi$-invariant subset. It follows that
$$\min\{1,HD(\tilde{\Lambda})\}\leq HD(\ell_{\varphi,f}(\widetilde{\Lambda})) \leq HD(m_{\varphi,f}(\widetilde{\Lambda})) \leq HD(f(\widetilde{\Lambda})) \leq \min\{1,HD(\tilde{\Lambda})\},$$
as we wanted to see.
\end{proof}

\begin{corollary}\label{max}
Given $\varphi\in\mathcal{U}^*$ and $f\in \mathcal{P}_{\varphi,\Lambda}$, one has
    $$\max L_{\varphi, f}=HD(\mathcal{L}_{\varphi,f})=\min \{1,HD(\Lambda)\}.$$
\end{corollary}

\subsection{A technical proposition}\label{A technical proposition} Throughout this subsection we will suppose $HD(\Lambda)\\ <1$. Fix $f\in\mathcal{R}_{\varphi, \Lambda}$ and take $X \subset \Lambda$, compact and $\varphi$-invariant. Observe that the same proof of \cite[Proposition 2.9]{CMM16} lets us conclude that for every $0<\eta<1$ there exists $\delta>0$ and a complete subshift $\Sigma(\mathcal{B}_u)\subset \Sigma_{\mathcal{B}}\subset\mathcal{A}^{\mathbb{Z}}$ associated to a finite set $\mathcal{B}_u$, of finite sequences  such that  
$$\Sigma(\mathcal{B}_u)\subset\Sigma_{\max f|_X-\delta} \quad \textrm{and} \quad D_u(\Lambda(\Sigma(\mathcal{B}_u)))>(1-\eta)D_u(X),$$
where $\Lambda(\Sigma(\mathcal{B}_u))$ denotes the subhorseshoe of $\Lambda$ associated to $\mathcal{B}_u$. We point out here that $\Lambda(\Sigma(\mathcal{B}_u))$ does not need to be contained in $X$.

For fixing ideas and for future use we will remember some facts about the proof of \cite[Proposition 2.9]{CMM16}: the construction of $\mathcal{B}_u$ depends on three combinatorial lemmas (2.13-2.15). In our case, to prove those lemmas, we take $r_0$ large given by (\ref{du}), such that  
\begin{equation} \label{limit}
 \left|\frac{\log N_u(X,r)}{r}-D_u(X)\right|<\frac{\tau}{2}D_u(X)   
\end{equation}
for all $r\in\mathbb{N}$, $r\geq r_0$ where $\tau=\eta/100$.

The alphabet $\mathcal{B}_u$ is obtained from the set$$\widetilde{\mathcal{B}}_u=\{\beta=\beta_1\dots\beta_k : \beta_j\in\ \mathcal{C}_u(X,r_0), \,\,\, \forall \, 1\leq j\leq k \,\, \textrm{ and } \,\, \pi^u(X)\cap I^u(\beta)\neq\emptyset\}$$ 
where $k=8 N_u(X,r_0)^2\lceil2/\tau\rceil$. 

Defining the notion of \emph{good position} for positions $j \in \{1,...,k \}$ (see Definition \ref{good-position} below) is shown that
most positions of most words of $\widetilde{\mathcal{B}}_u$ are good and for that set of words, say $\mathcal{E}$, we can find natural numbers $1\leq s_1\leq \dots\leq s_{3N_0^2}\leq k$, ($N_0=N_u(X,r_0)$) with 
$$s_{m+1}-s_m\geq 2\lceil2/\tau\rceil \quad \mbox{for} \quad 1\leq m<3N_0^2$$ 
and words $\widehat{\beta}_{s_1}, \widehat{\beta}_{s_1+1}, \dots, \widehat{\beta}_{s_{3N_0^2}}, \widehat{\beta}_{s_{3N_0^2}+1}\in\ \mathcal{C}_u(X,r_0)$ such that the set $\mathcal{P}$ of words in $\mathcal{E}$ with $ s_m, s_{m}+1 $ good positions and $\beta_{s_m}=\widehat{\beta}_{s_m}, \beta_{s_m+1}=\widehat{\beta}_{s_m+1}$ for $ 1\leq m<3N_0^2$ has cardinality $\abs{\mathcal{P}}>N_0^{(1-2\tau)k}.$ 

Then is proved that there are $1\leq p_0<q_0\leq 3N_0^2$ such that $\widehat{\beta}_{s_{p_0}}=\widehat{\beta}_{s_{q_0}}$, $\widehat{\beta}_{s_{p_0+1}}=\widehat{\beta}_{s_{q_0+1}}$ and the cardinality of $\mathcal{B}_u=\pi_{p_0,q_0}(\mathcal{P})$ is
$$\abs{\mathcal{B}_u} > N_0^{(1-10\tau)(s_{q_0}-s_{p_0})},$$
where
$$\pi_{p_0,q_0}: \mathcal{P}\to \mathcal{C}_u(X,r_0)^{s_{q_0}- s_{p_0}} \quad \textrm{is the projection} \quad 
\pi_{p_0,q_0}(\beta_1\dots\beta_k)=(\beta_{s_{p_0+1}},\dots,\beta_{s_{q_0}})$$
obtained by cutting a word $\beta_1\dots\beta_k\in \mathcal{P}$ at the positions $s_{p_0}$ and $s_{q_0}$ and discarding the words $\beta_j$ with $j\leq s_{p_0}$ and $j>s_{q_0}$.

Using the conclusion on the cardinality of $\mathcal{B}_u$ is shown that $D_u(\Lambda(\Sigma(\mathcal{B}_u)))>\\(1-\eta)D_u(X)$ and using that $s_{p_0}$, $s_{p_0}+1$, $s_{q_0}$ and $s_{q_0}+1$ are good positions for words in $\mathcal{P}$ that $\Sigma(\mathcal{B}_u)\subset\Sigma_{\max f|_X-\delta}$.

Even more, the proof of that proposition gives us the next formula:
\begin{equation}\label{delta}
\delta=\min\{\delta^1, \delta^2, \\ \delta^3, \delta^4\},
\end{equation}
where if $\gamma_1= \widehat{\beta}_{s_{p_0+1}}= a_1\dots a_{\widehat{m}_1}$, $\beta_{s_{p_0}+2}\dots\beta_{s_{q_0}-1}=b_{1}\dots b_{\widehat{m}}$ and $\gamma_2=\widehat{\beta}_{s_{q_0}}=d_{1}\dots d_{\widehat{m}_2}$ then
\begin{itemize}
\item $\delta^1=c_3\cdot\min\limits_{\gamma_1 b_1\dots b_{\widehat{m}}\gamma_2\in\mathcal{B}_u}\,\,\, \min\limits_{1\leq j\leq \widehat{m}-1}\,\,\, |I^u(b_j\dots b_{\widehat{m}}\gamma_2)| $ 
\item $ \delta^2=c_3\cdot\min\limits_{\gamma_1 b_1\dots b_{\widehat{m}}\gamma_2\in\mathcal{B}_u}\,\,\, \min\limits_{1\leq j\leq \widehat{m}-1}\,\,\, |I^s((\gamma_1b_1\dots b_{j-1})^T)|$ 
\item $\delta^3=c_3\cdot\min\limits_{\gamma_1 b_1\dots b_{\widehat{m}}\gamma_2\in\mathcal{B}_u}\,\,\, \min\limits_{1\leq \ell\leq \widehat{m}_1-1}\,\,\, |I^s((\gamma_2 a_1\dots a_{\ell})^T)|$
\item $\delta^4=c_3\cdot\min\limits_{\gamma_1 b_1\dots b_{\widehat{m}}\gamma_2\in\mathcal{B}_u}\,\,\, \min\limits_{1\leq \ell\leq \widehat{m}_1-1}\,\,\, |I^u(d_{\ell-\widehat{m_1}-\widehat{m}+1}\dots d_{\widehat{m}_2}\gamma_1)|
$\end{itemize}
and $c_3$ is a positive constant that only depends on the function $f$ and $\varphi$. 

We will give a more precise estimate of the value of $\delta=\delta(\eta,X)$ and show some uniformity property of it. We also want to better describe the horseshoe $\Lambda ^u(X)=\Lambda(\Sigma(\mathcal{B}_u))$ obtained before. To do this, let us consider for $n\in \mathbb{N}$ the set $C(X,n)$ of admissible finite words $\theta$ of the form $\theta=(a_{-n},\dots,a_0,\dots, a_{n})$, such that the rectangle $R(a_{-n},\dots,a_0,\dots, a_{n};0)=\bigcap\limits_{j=-n}^{n}\varphi^{-j}(R_{a_j})$ has nonempty intersection with $X$. Also, given $\epsilon>0$ define $n(\epsilon)=\min\{n\in\mathbb{N}:\forall \theta\in C(\Lambda,n), \  \text{diam}(R(\theta;0))\leq \epsilon/2 \}$ where $\text{diam}(R(\theta;0))$ denotes the diameter of the set $R(\theta;0)$. 

\begin{proposition}\label{mcu2}

Given $\epsilon >0$ and $c_0 >0$ there exists a constant $\delta=\delta(\epsilon, c_0)>0$ such that if $X$ is a compact $\varphi$-invariant subset of $\Lambda$ that satisfies $D_u(X) \geq c_0$, then we can find some subhorseshoe $\Lambda ^u(X)$ of $\Lambda$ such that 
$$D_u(\Lambda ^u(X))>(1-\epsilon)D_u(X)\ \text{and}\ \Lambda ^u(X)\subset \Lambda_{\max f|_X-\delta}.$$ 
Furthermore, for every $x\in \Lambda ^u(X)$ the set
 \begin{eqnarray*}
X_\epsilon(x)=\{n\in \mathbb{Z}: \exists \theta\in C(X,n(\epsilon))\ \mbox{such that}\ \varphi^n(x)\in R(\theta;0)\}
\end{eqnarray*} 
is neither bounded below nor bounded above.
\end{proposition}

\begin{proof}
Take $X \subset \Lambda$, compact and $\varphi$-invariant as in the statement of the proposition. It is clear from the construction given of $\mathcal{B}_u$ and from the fact that
$$s_{q_0}-s_{p_0}\geq 2\lceil2/\tau\rceil (q_0 -p_0) \geq 2\lceil2/\tau\rceil=2\lceil200/\eta\rceil$$
that for $\eta=\eta(\epsilon) < \epsilon$ small enough and $x \in \Lambda ^u(X)=\Lambda(\Sigma(\mathcal{B}_u))$, the set $X_\epsilon(x)$ is neither bounded below nor bounded above. Also, because $\Lambda ^u(X)\subset \Lambda_{\max f|_X-\delta}$, the proposition will be proved if we can choose $\delta$ only depending on $\eta$ and $c_0$.

Without loss of generality, consider $0<\eta < \min \{c_0, 5000/ (c_2\log \abs{\mathcal{A}}), 3\lambda_1, \kappa \},$
where $\kappa >0$ is such that the maps $x\mapsto e^{e^x}-8e^{2\alpha_1x+2\alpha_2}\cdot x^2$  and $x\mapsto e^{e^x}-8\log x\cdot e^{2\alpha_1x+2\alpha_2} \cdot x(\alpha_1 x+\alpha_2) $  are positive if $x > 1/\kappa^2$. Here $\lambda_1$ is given in \ref{bdp2}, $\alpha_1,\alpha_2$ in \ref{expN} and $c_2$ in Lemma \ref{c_2}.

The crucial observation here is that in the proof sketched above (without the dimension estimate) we can replace the conditions on $r_0$ (and $k$), given by Equation (\ref{limit}), by the assumption that $r_0>\lceil\frac{4(c_1+1)\log|\mathcal{A}|^{c_2}}{c_0\tau^2}\rceil$ and $k=8  N_u(X,r_0)^2\lceil2/\tau\rceil$ satisfy the inequality
\begin{equation}\label{condition}
\frac{\log N_u(X,r_0)}{r_0} < (1+\frac{\tau}{2})\frac{\log N_u(X,k(r_0-c_1))}{k(r_0-c_1)},
\end{equation}
where $c_1$ comes from the bounded distortion property as in Equation (\ref{bdp1}). This is because, if we multiply this inequality by $(1-\tau)r_0k$, we have by our choice of $r_0$ that 
\begin{eqnarray*}\log N_u(X,r_0)^{(1-\tau)k} &<& (1-\tau)(1+\frac{\tau}{2})\frac{r_0}{r_0-c_1}\log N_u(X,k(r_0-c_1))\\ &=&(1-\frac{\tau}{2}-\frac{\tau^2}{2})(1+ \frac{c_1}{r_0-c_1})\log N_u(X,k(r_0-c_1))\\ &<& (1-\frac{\tau}{2})(1+ \frac{c_1}{r_0-c_1})\log N_u(X,k(r_0-c_1)) \\&<&  (1-\frac{\tau}{2})(1+\frac{\tau^2}{1-\tau^2})\log N_u(X,k(r_0-c_1))\\ &<&  (1-\frac{\tau}{2}) (1+\frac{\tau}{2})\log N_u(X,k(r_0-c_1)) \\ &=&\log N_u(X,k(r_0-c_1))^{1-\frac{\tau^2}{4}}
\end{eqnarray*}
also, given any $r\geq r_0$ we have by definition of $D_u(X)$
\begin{equation}\label{dime}
(1-\frac{\tau}{2})D_u(X)\leq D_u(X)-\frac{\tau}{2}c_0\leq D_u(X)-\frac{\log\abs{\mathcal{A}}^{c_2}}{r}\leq \frac{\log N_u(X,r)}{r} 
\end{equation}
which implies for $r=k(r_0-c_1)$ that
$$\log2<\log\abs{\mathcal{A}}^{c_2}< \frac{\tau^2}{4}r_0c_0 \leq\frac{\tau^2}{4}(1-\frac{\tau}{2})k(r_0-c_1)D_u(X)\leq\frac{\tau^2}{4}\log N_u(X,k(r_0-c_1)).$$
From the previous inequalities, we conclude that $2N_u(X,r_0)^{(1-\tau)k}<N_u(X,k(r_0-c_1))$ which is precisely the necessary condition to obtain \cite[Equation 2.4, Lemma 2.13]{CMM16} and the claims in other parts of the proof of the lemmas that use the assumptions that $r_0$ and $k$ are large, are satisfied provided $r_0> \lceil\frac{4(c_1+1)\log|\mathcal{A}|^{c_2}}{c_0\tau^2}\rceil$.

On the other hand, given any regular Cantor set $(K,\psi)$ with Markov partition $\mathcal{P}=\{I_1, \dots, I_k\}$ if we define inductively $\mathcal{R}_1=\mathcal{P}$ and for $n\geq 2$, $\mathcal{R}_n$ as the set of connected components of $\psi^{-1}(J)$, $J\in \mathcal{R}_{n-1}.$ And also, for each $R\in \mathcal{R}_n$ we denote by
$$ \lambda_{n, R}=\inf \abs{(\psi^n)'|_R} \ \ \text{and}\ \ \Lambda_{n, R}=\sup \abs{(\psi^n)'|_R},$$ 
the bounded distortion property shows the existence of some $a=a(K)\geq1$, such that $\Lambda_{n,R}\leq a . \lambda_{n,R}$, for all $n\geq 1$. Even more, it is well known that for any such $K$, $D(K)=HD(K)$ where $D(K)$ denotes the limit capacity of $K$ (cf. \cite[Chapter 4]{PT93}). Indeed, it follows from the proof of this result that for the sequences $\{ \alpha_n\}_{n\in \mathbb{N}}$ and $\{ \beta_n\}_{n\in \mathbb{N}}$ given by
\begin{equation}\label{beta_n}
  \sum \limits_{R\in \mathcal{R}_n} \left (\frac{1}{\Lambda_{n, R}}\right )^{\alpha_n}=1=\sum \limits_{R\in \mathcal{R}_n} \left (\frac{1}{\lambda_{n, R}}\right )^{\beta_n},   
\end{equation}
when $\psi$ is a full Markov map i.e., $\psi(K\cap I_j)=K$ for $1\leq j\leq k$, one has
\begin{equation}\label{alfabeta}
 \alpha_n \leq HD(K)=D(K)\leq \beta_n   
\end{equation}
and if $n\geq \log a/\log \lambda$, where $\lambda=\lambda(K)=\inf\abs{\psi'}>1$
\begin{equation}\label{Palistimates}
    \beta_n-\alpha_n\leq \frac{\log a\cdot HD(K)}{n\log \lambda-\log a}.
\end{equation}

Now, consider the Cantor set $K^u(\Sigma(\mathcal{B}_u))$ consisting of points of $K^u$ whose trajectory under $\psi_u$ follows an itinerary obtained from the concatenation of words in the alphabet $\mathcal{B}_u$. This Cantor set is $C^{1+\alpha}$-dynamically defined associated to certain iterates of $\psi_u$ on the intervals $I^u(\beta)$ with $\beta \in \mathcal{B}_u$. 
If $r(\eta,\Lambda)\in \mathbb{N}$ is such that given $r_0 \geq r(\eta,\Lambda)$ one has for any complete subshift associated to a finite alphabet $\mathcal{B}_u=\mathcal{B}_u(r_0)$ of finite words as before that $\lambda=\lambda(K^u(\Sigma(\mathcal{B}_u)))$ is big  (we can take $a=a(K^u(\Sigma(\mathcal{B}_u)))=a(K^u(\Lambda))$), then by (\ref{alfabeta}) and (\ref{Palistimates}) 

$$\beta_1-\alpha_1\leq \frac{\tau}{2} HD(K^u(\Sigma(\mathcal{B}_u)))\leq \frac{\tau}{2}\beta_1.$$
Using this, (\ref{beta_n}) and (\ref{alfabeta}) we obtain 
$$HD(K^u(\Sigma(\mathcal{B}_u)))\ge \alpha_1 \geq \left (1-\dfrac{\tau}{2}\right)\beta_1 \geq \left (1-\dfrac{\tau}{2}\right) \dfrac{| \mathcal{B}_u|}{- \log (\min\limits_{\alpha\in \mathcal{B}_u} |I^u(\alpha)|)}$$
which is the equation used in \cite{CMM16} (together with (\ref{dime})) to obtain the dimension estimate
$$D_u(\Lambda(\Sigma(\mathcal{B}_u)))>(1-\eta)D_u(X).$$

Following the observations described above, we will try to find $r_0$ large enough, satisfying (\ref{condition}). For this, define the sequence $\{p_n\}$ as follows: $p_0=\max \{ \lceil\frac{4(c_1+1)\log|\mathcal{A}|^{c_2}}{c_0\tau^2}\rceil,\\ r(\eta,\Lambda) \}$ and for $n\geq0$ put 
$$p_{n+1}=8 N_u(X,p_n)^2\lceil2/\tau\rceil(p_n -c_1).$$
We claim that, for some integer $0\leq s_0<(1+\frac{2}{\tau})\log \frac{4(\alpha_1 + \alpha_2+1)}{\eta}$ one has 
 $$\frac{\log N_u(X,p_{s_0})}{p_{s_0}} < (1+\frac{\tau}{2})\frac{\log N_u(X,p_{s_0+1})}{p_{s_0+1}}= (1+\frac{\tau}{2})\frac{\log N_u(X,k(p_{s_0}-c_1))}{k(p_{s_0}-c_1)},$$
with 
$k=8 N_u(X,p_{s_0})^2\lceil2/\tau\rceil$. 

Indeed, if it is not the case, then for 
$0\leq n<(1+\frac{2}{\tau})\log\frac{4(\alpha_1 + \alpha_2+1)}{\eta}$, we have 

$$\frac{\log N_u(X,p_{n+1})}{p_{n+1}}\leq (1+\frac{\tau}{2})^{-1}\frac{N_u(X,p_n)}{p_n} $$ 
and then, for $M=\lceil  (1+\frac{2}{\tau})\log\frac{4(\alpha_1 + \alpha_2+1)}{\eta}  \rceil$ we would have 
$$\frac{\log N_u(X,p_M)}{p_M} \leq (1+\frac{\tau}{2})^{-M} \cdot\frac{\log N_u(X,p_0)}{p_0} < \frac{\eta}{4(\alpha_1 + \alpha_2+1)}\frac{\log N_u(X,p_0)}{p_0}$$ 
because 
$$(1+\frac{\tau}{2})^{-M} \leq ((1+\frac{\tau}{2})^{-(1+\frac{2}{\tau})})^{\log\frac{4(\alpha_1+\alpha_2+1)}{\eta}} < e^{-\log\frac{4(\alpha_1+\alpha_2+1)}{\eta}}=\frac{\eta}{4(\alpha_1+\alpha_2+1)}.$$
And so, by (\ref{expN})
$$\frac{\log N_u(X,p_M)}{p_M} \leq \frac{\eta}{4(\alpha_1 + \alpha_2)}\frac{\log N_u(X,p_0)}{p_0} \leq \frac{\eta}{4(\alpha_1 + \alpha_2)}\frac{\alpha_1.p_0 + \alpha_2}{p_0}<\frac{\eta}{2}.$$ 
But this is a contradiction because by (\ref{dime})
$$\frac{\eta}{2} < (1-\frac{\tau}{2})c_0 \leq (1-\frac{\tau}{2})D_u(X) \leq \frac{\log N_u(X,p_M)}{p_M}.$$

Therefore, by taking $r_0=p_{s_0}$ and $k=8 N_u(X,r_0)^2\lceil2/\tau\rceil$, the argument for the construction of $\mathcal{B}_u$ works and then, because of (\ref{bdp2}), (\ref{words}) and (\ref{delta}), we have
\begin{equation} \label{2.17}
 \delta \geq c_3e^{-c_1}\lambda_1^{\widehat{m}_1+\widehat{m}_2+\widehat{m}} \geq c_3e^{-c_1}\lambda_1^{k\cdot \max \{ \abs{\alpha}:\alpha \in \mathcal{C}_u(X,r_0)   \}} \geq c_3e^{-c_1}\lambda_1^{k\cdot(\alpha_1r_0+\alpha_2)}.  
\end{equation}
We will now give an explicit positive lower bound for $\delta$ in terms of $\eta$. In order to do that, we define recursively, for each integer $n \geq 0$ and $x \in \mathbb{R}$, the function $\mathcal{T}(n,x)$ by $\mathcal{T}(x,0)=x$, $\mathcal{T}(x,n+1)=e^{\mathcal{T}(x,n)}$. By (\ref{expN}), we have for $n \geq0$
$$p_{n+1}=8 N_u(X,p_n)^2\lceil2/\tau\rceil(p_n -c_1)<  8e^{2\alpha_1p_n+2\alpha_2}\cdot p_n^2 < e^{e^{p_n}}, $$
since $p_n \geq p_0 > \lceil2/\tau^2\rceil$ and $p_n > 1/\kappa^2$.  Therefore $r_0=p_{s_0} < \mathcal{T}(p_0,2s_0)$ and
\begin{eqnarray*}\label{lala}
\log \lambda_1^{-1}\cdot k(\alpha_1 r_0+\alpha_2) &=& 8\log \lambda_1^{-1}\cdot N_u(X,r_0)^2\lceil2/\tau\rceil (\alpha_1 r_0+\alpha_2)\\ &<&   8\log r_0\cdot e^{2\alpha_1r_0+2\alpha_2} \cdot r_0(\alpha_1 r_0+\alpha_2) < e^{e^{r_0}} 
\end{eqnarray*} 
so, by (\ref{2.17})
\begin{equation} \label{2.19}
  \delta \geq c_3e^{-c_1}e^{\log \lambda_1 \cdot k(\alpha_1r_0+\alpha_2)}> c_3e^{-c_1}e^{-e^{e^{r_0}}} > \frac{c_3e^{-c_1}}{\mathcal{T}(p_0,2s_0+3)}.
\end{equation}
As $p_0=\max \{ \lceil\frac{40000(c_1+1)\log|\mathcal{A}|^{c_2}}{c_0\eta^2}\rceil, r(\eta,\Lambda) \}$ and $s_0<(1+\frac{2}{\tau})\log \frac{4(\alpha_1+\alpha_2+1)}{\eta}= (1+\frac{200}{\eta})\log \frac{4(\alpha_1+\alpha_2+1)}{\eta}$, we have by (\ref{2.19})
$$\delta > \frac{c_3e^{-c_1}}{\mathcal{T}(p_0,2s_0+3)} = \frac{c_3e^{-c_1}}{\mathcal{T}(\max \{ \lceil\frac{40000(c_1+1)\log|\mathcal{A}|^{c_2}}{c_0\eta^2}\rceil, r(\eta,\Lambda) \},\lceil \frac{201}{\eta}\log \frac{4(\alpha_1+\alpha_2+1)}{\eta}\rceil)},$$
that finishes the proof of the proposition.
\end{proof}
Now, if we suppose that $D_s(X)\geq c_0$, given $\epsilon>0$ we can construct, as before, some complete subshift $\Sigma(\mathcal{B}_s)$ such that $\Lambda(\Sigma(\mathcal{B}_s))$ has similar properties as $\Lambda^u(X)=\Lambda(\Sigma(\mathcal{B}_u))$. Then, we immediately have
\begin{corollary}\label{uniform}
Given $\epsilon >0$ and $c_0 >0$ there exists a constant $\delta=\delta(\epsilon, c_0)>0$ such that if $X$ is a compact $\varphi$-invariant subset of $\Lambda$ such that the limit capacities $D_u(X)$ and $D_s(X)$ satisfy both $D_u(X), D_s(X) \geq c_0$, then there are subhorseshoes $\Lambda ^s(X)$ and  $\Lambda ^u(X)$ of $\Lambda$ such that 
$$D_u(\Lambda ^u(X))>(1-\epsilon)D_u(X), \quad D_s(\Lambda ^s(X))>(1-\epsilon)D_s(X)$$ 
and 
  $$\Lambda ^u(X)\cup \Lambda ^s(X) \subset \Lambda_{\max f|_X-\delta}.$$ 
Furthermore, for every $x\in \Lambda ^u(X) \cup \Lambda ^s(X)$ the set
\begin{eqnarray*}
X_\epsilon(x)=\{n\in \mathbb{Z}: \exists \theta\in C(X,n(\epsilon))\ \mbox{such that}\ \varphi^n(x)\in R(\theta;0)\}
\end{eqnarray*} 
is neither bounded below nor bounded above.
\end{corollary}

\subsection{First accumulation point of the Lagrange spectrum}
In this subsection, we show the existence of the first accumulation point of the Lagrange spectrum and show that it is exactly at that point where the map $L_{\varphi, f}$ begins to be positive. In what follows, we will use the following result from \cite{GCD}:

\begin{lemma}\label{L1}
Given $\varphi\in \mathcal{U},$ any subhorseshoe $\tilde{\Lambda}\subset \Lambda$, $f\in C^1(S,\mathbb{R})$ and $t\in\mathbb{R}$, one has
$$\ell_{\varphi,f}(\tilde{\Lambda})\cap (-\infty,t)= \bigcup \limits_{s<t} \ell_{\varphi,f}(\tilde{\Lambda}_s).$$
In particular
$$L_{\varphi, f}(t)=\sup \limits_{s <t} HD(\ell_{\varphi,f}(\Lambda_s))=\lim \limits_{s \to\ t^-} HD(\ell_{\varphi,f}(\Lambda_s)).$$
\end{lemma}
From this we get 
\begin{equation}\label{L}
    L_{\varphi, f}(t)=\sup \limits_{s <t} HD(\ell_{\varphi,f}(\Lambda_s)) \leq HD(\ell_{\varphi,f}(\Lambda_{t})) \leq HD(f(\Lambda_{t}))\leq HD(\Lambda_{t}).
\end{equation}

 \begin{proposition}\label{first}
    Take $\varphi \in \mathcal{U}^*$ and $f\in \mathcal{P}_{\varphi,f}$. Then 
     $$\mathcal{L}^{'} _{\varphi, f}=\{x: x\ \textrm{is an accumulation point of}\ \mathcal{L}_{\varphi, f}\}\neq \emptyset$$ and $c_{\varphi,f}=\min L^{'} _{\varphi, f}$.
 \end{proposition}

\begin{proof}
First, by Proposition \ref{lagrange1} 
$$HD(\mathcal{L}_{\varphi,f})=HD(\ell_{\varphi,f}(\Lambda))=\min \{1,HD(\Lambda)\}>0,$$
then, $\mathcal{L}_{\varphi,f}$ cannot be finite and as $\mathcal{L}_{\varphi,f}\subset f(\Lambda)$, it must be true that $\mathcal{L}^{'} _{\varphi, f}\neq \emptyset$.

Let $c^*_{\varphi,f}=\min L^{'} _{\varphi, f}$. Given $\epsilon>0$, it is clear that $L_{\varphi,f}(c^*_{\varphi,f}-\epsilon)=0$ because $\mathcal{L}_{\varphi,f}\cap  (-\infinity,c^*_{\varphi,f}- \epsilon)$ is finite. On the other hand, take an injective sequence $(y_n)_{n\in \mathbb{N}}=(\ell_{\varphi,f}(x_n))_{n\in \mathbb{N}}\subset \mathcal{L}_{\varphi,f}$ such that $\lim \limits_{n\rightarrow \infinity}y_n=c^*_{\varphi,f}$ and consider $N\in \mathbb{N}$ big enough such that for two elements $x,y\in \Lambda$ if their kneading sequences coincide in the central block (centered at the zero position) of size $2N+1$ then $\abs{f(x)-f(y)}<\epsilon/6$.

Take first $n_0\in \mathbb{N}$ large so that $\abs{\ell_{\varphi,f}(x_n)-c^*_{\varphi,f}}<\epsilon/6$ for $n\geq n_0$ and there are infinitely many $j\in \mathbb{N}$ such that $\abs{f(\varphi^j(x_n))-c^*_{\varphi,f}}<\epsilon/6$. Given such a pair $(j,n)$, consider the finite sequence with $2N+1$ terms $S(j,n)=(b_{j-N}^{(n)},b_{j-N+1}^{(n)},\cdots, b_{j}^{(n)}, \cdots, b_{j+N}^{(n)})$ where $\Pi^{-1}((b_j^{(n)})_{j\in \mathbb{Z}})=x_n$. There is a sequence $S$ such that for infinitely many values of $n$, $S$ appears infinitely many times as $S(j,n)$; i.e., there are $j_1(n)<j_2(n)<\cdots$ with $\lim \limits_{i\rightarrow \infinity}(j_{i+1}(n)-j_i(n))=\infinity$ and $S(j_i(n),n)=S$ for all $i\geq 1$ and for all $n$ in some infinite set $A\subset \mathbb{N}$.

Consider the sequences $\beta(i,n)$ for $i\geq 1$, $n\in A$ given by
$$\beta(i,n)=(b_{j_i(n)+N+1}^{(n)},b_{j_i(n)+N+2}^{(n)},\cdots, b_{j_{i+1}(n)+N}^{(n)}).$$
Taking $n_1, n_2\in A$ distinct and $r=r(n_1,n_2)$ large enough such that for $j\geq r$, $f(\varphi^j(x_{n_1}))<\ell_{\varphi,f}(x_{n_1})+\epsilon/6$ and $f(\varphi^j(x_{n_2}))<\ell_{\varphi,f}(x_{n_2})+\epsilon/6$. There are $i_1\geq r$ and $i_2 \geq r$ for which there is no a sequence $\gamma$ such that $\beta(i_1,n_1)$ and $\beta(i_2,n_2)$ are concatenations of copies of $\gamma$, otherwise $y_{n_1}=y_{n_2}$  because for $n\in A$ 
$$\Pi(x_n)=(\cdots, b_1^{(n)},\cdots b_{j_1(n)+N}^{(n)},\beta(1,n),\beta(2,n),\cdots ,\beta(m,n), \cdots).$$

This implies that, by taking
$$C=\{\beta(i_1,n_1)\beta(i_2,n_2), \beta(i_2,n_2)\beta(i_1,n_1)\},$$
we have $\Sigma(C)$ is a complete subshift and for  $x\in \Lambda(\Sigma(C))=\Lambda_{C}$ (the subhorseshoe associated to $\Sigma(C)$) we have $m_{\varphi,f}(x)<c^*_{\varphi,f}+\epsilon/2$. Indeed, for every $k\in \mathbb{Z}$, the kneading sequence of $\varphi^k(x)$ coincides in the central block of size $2N+1$ with the kneading sequence of $\varphi^l(x_{\theta})$ where $\theta$ is either $n_1$ or $n_2$ and $l\geq r$. So,
$$f(\varphi^k(x))<f(\varphi^l(x_{\theta}))+\frac{\epsilon}{6}<\ell_{\varphi,f}(x_{\theta})+\frac{\epsilon}{3}< c^*_{\varphi,f}+\frac{\epsilon}{2}.$$

Therefore, using one more time Proposition \ref{lagrange1} and Lemma \ref{L1} we conclude 
$$0<\min \{1,HD(\Lambda_{C})\}=HD(\ell_{\varphi,f}(\Lambda_{C}))\leq HD(\ell_{\varphi,f}(\Lambda_{c^*_{\varphi,f}+\epsilon/2}))\leq L_{\varphi, f}(c^*_{\varphi,f}+\epsilon).$$
Then, by definition $c^*_{\varphi,f}=c_{\varphi,f}$, which ends the proof of the proposition.
\end{proof}
\begin{corollary}If $HD(\Lambda)<1$ one has
 $$c_{\varphi,f}=\max\{t\in \mathbb{R}:HD(\Lambda_t)=0\}.$$
\end{corollary}
\begin{proof}
It follows from the previous proposition and (\ref{L}) that $0<L_{\varphi, f}(c_{\varphi,f}+\epsilon)\leq HD(\Lambda_{c_{\varphi,f}+\epsilon}).$
Now, if $HD(\Lambda_{c_{\varphi,f}})>0$ then by (\ref{Du}), $D_u(\Lambda_{c_{\varphi,f}})>0$ (also $D_s(\Lambda_{c_{\varphi,f}})>0$), and by Proposition \ref{mcu2} we can find some horseshoe $\tilde{\Lambda}\subset\Lambda_{c_{\varphi,f}-\delta}$ for some $\delta>0$ and arguing as before, we get the contradiction $L_{\varphi, f}(c_{\varphi,f}-\delta/2)>0.$   
\end{proof}
\begin{remark}
This corollary remains true if $HD(\Lambda)\geq 1$ because of \cite[Proposition 1]{GCD} lets us also show the existence of $\tilde{\Lambda}$ and $\delta>0$ as before. 
\end{remark}
\begin{corollary}\label{c}
If $HD(\Lambda)<1$ then $L_{\varphi, f}$ is continuous in $c_{\varphi,f}.$
\end{corollary}
\begin{proof}
Suppose $\lim \limits_{t \to\ c_{\varphi,f}^+} HD(\Lambda_t)=h>0$, then by (\ref{Du}), for $t>c_{\varphi,f}$ one has $D_u(\Lambda_t)\geq h/(1+\tilde{C}).$ On the other hand, Proposition \ref{mcu2} let us find some $\delta=\delta(\frac{1}{2},\frac{h}{1+\tilde{C}})>0$ such that for any $t>c_{\varphi,f}$ we can find some horseshoe $\Lambda^u(\Lambda_t)\subset \Lambda_{t-\delta}$ (the other conclusions of the proposition are not necessary here). By applying this to $t=c_{\varphi,f}+\delta/2$, we get the contradiction $0<HD(\Lambda^u(\Lambda_{c_{\varphi,f}+\delta/2}))\leq HD(\Lambda_{c_{\varphi,f}-\delta/2})$. Then 
$$0=L_{\varphi, f}(c_{\varphi,f})\leq \lim \limits_{t \to\ c_{\varphi,f}^+}L_{\varphi, f}(t)\leq\lim \limits_{t \to\ c_{\varphi,f}^+} HD(\Lambda_t)=0,$$ 
as we wanted to see.
\end{proof}
\begin{remark}\label{rr}
This corollary also holds when $HD(\Lambda)\geq 1$ because as we will see later, before $\tilde{c}_{\varphi,f}$, it is true some expression of the type $L_{\varphi, f}=\max\limits_{i}L_i$, where the functions $L_i$ are defined as $L_{\varphi, f}$ but are associated to horseshoes with Hausdorff dimension less than $1$.     
\end{remark}

\subsection{Geometric consequences of having a discontinuity}
In this subsection, we show how to associate to each discontinuity the pair of subhorseshoes described in the introduction of the section.

Take $\varphi \in \mathcal{U}^*$ with $HD(\Lambda)<1$, $f\in \mathcal{P}_{\varphi,\Lambda}$ and suppose $t_0\in \mathbb{R}$ is a discontinuity of the map  $t \mapsto L_{\varphi,f}(t)= HD(\mathcal{L}_{\varphi, f}\cap (-\infinity,t))$. So, there exists an $a>0$ such that
\begin{equation}\label{a-L}
L_{\varphi, f}(q)+a < L_{\varphi, f}(s) \ \mbox{for} \ q\leq t_0<s.    
\end{equation}
 By Corollary \ref{c} and (\ref{L}) we have $0<L_{\varphi, f}(t_0)\leq HD(\Lambda_{t_0}),$
then $D_u(\Lambda_{t_0})>0$ and by Proposition \ref{mcu2}, we can find some horseshoe $\Lambda^0 \subset \Lambda_{t_0}$.
For $0<\epsilon < a/2$ and $c_0=HD(\Lambda^0)/(\tilde{C}+1)>0$, where $\tilde{C}$ si given in (\ref{conservative}), take $\delta=\delta(\epsilon/2k, c_0) < \epsilon$ as in Corollary \ref{uniform} where $k>1$ is a Lipschitz's constant for $f$. 

Let us consider for $t\in \mathbb{R}$ and $n\in \mathbb{N}$ the set $C(\Lambda_t,n)$, defined in Subsection \ref{A technical proposition}. By compactness, one has 
$$C(\Lambda_{t_0},n)=\bigcap_{t> t_0} C(\Lambda_t,n).$$
In particular, for each $n$, there exists $t(n)>t_0$ such that for $t_0 < t \leq t(n)$
$$C(\Lambda_t,n)=C(\Lambda_{t_0},n).$$ 
Take then, $n=n(\delta/2k)=\min\{r\in\mathbb{N}:\forall \theta\in C(\Lambda,r), \  \text{diam}(R(\theta;0))\leq \delta/4k\}$ and consider the maximal invariant set 
$$P=M(C(\Lambda_{t_0},n))=\bigcap \limits_{m \in \mathbb{Z}} \varphi ^{-m}(\bigcup \limits_{\theta\in C(\Lambda_{t_0},n)}  R(\theta;0))=\bigcap \limits_{m \in \mathbb{Z}} \varphi ^{-m}(\bigcup \limits_{\theta\in C(\Lambda_t,n)}  R(\theta;0))$$
for $t_0 < t \leq t(n)$. Observe that given $x\in P$ and $m\in \mathbb{Z}$ if $y\in \Lambda_{t_0}$ belongs to the same rectangle $R(\theta;0)$ as $\varphi^m(x)$ for some $\theta\in C(\Lambda_{t_0},n)$ then
$$f(\varphi^m(x))\leq f(\varphi^m(x))-f(y)+t_0  \leq k\cdot d(\varphi^m(x),y)+t_0 \leq k\cdot \frac{\delta}{4k}+t_0 < \frac{\delta}{2} + t_0$$
and so $P \subset \Lambda_{t_0+\delta /2}$.

Remember that for any subhorseshoe $\tilde{\Lambda} \subset \Lambda$, being locally maximal, we have 
\begin{equation}\label{omega}
 \bigcup \limits_{y\in \tilde{\Lambda}}W^s(y)=W^s(\tilde{\Lambda})= \{y\in S: \lim \limits_{n \to \infty}d(\varphi^n(y),\tilde{\Lambda})=0\}.   
\end{equation}
Now, by Proposition \ref{appendix}, the set $P$ admits a decomposition $P=\bigcup \limits_{i\in \mathcal{I}} \tilde{\Lambda}_i$ where $\mathcal{I}$ is a finite index set and for any $i\in \mathcal{I}$,\ $\tilde{\Lambda}_i$ is a subhorseshoe or a transient set. In particular, given $i_1\in \mathcal{I}$ we can find $i_2\in \mathcal{I}$ such that $\tilde{\Lambda}_{i_2}$ is a subhorseshoe with $\omega(x)\subset \tilde{\Lambda}_{i_2}$ for every $x\in\tilde{\Lambda}_{i_1}$; and from this and (\ref{omega}), it follows that $\ell_{\varphi,f}(x)=\ell_{\varphi,f}(y)$ for some $y\in\tilde{\Lambda}_{i_2}$. We conclude then 
$$\ell_{\varphi,f}(P)=\bigcup \limits_{i\in \mathcal{I}} \ell_{\varphi,f}(\tilde{\Lambda}_i)=\bigcup \limits_{\substack{i\in \mathcal{I}: \ \tilde{\Lambda}_i \ is\\ horseshoe }} \ell_{\varphi,f}(\tilde{\Lambda}_i) \cup \bigcup \limits_{\substack{i\in \mathcal{I}: \ \tilde{\Lambda}_i\\ is \ orbit }} \ell_{\varphi,f}(\tilde{\Lambda}_i)$$ and by Proposition \ref{lagrange1}
\begin{eqnarray*}
HD(\ell_{\varphi,f}(P))&=&HD( \bigcup \limits_{\substack{i\in \mathcal{I}: \ \tilde{\Lambda}_i \ is\\ horseshoe }} \ell_{\varphi,f}(\tilde{\Lambda}_i))=\max \limits_{\substack{i\in \mathcal{I}: \ \tilde{\Lambda}_i \ is\\ horseshoe }} HD(\ell_{\varphi,f}(\tilde{\Lambda}_i))=\max \limits_{\substack{i\in \mathcal{I}: \ \tilde{\Lambda}_i \ is\\ horseshoe }} HD(\tilde{\Lambda}_i).
\end{eqnarray*}

Let $\tilde{\Lambda}_{i_0}$ with $HD(\ell_{\varphi,f}(P))=HD(\tilde{\Lambda}_{i_0})$. As $\Lambda^0 \subset P$, by (\ref{Ds}) and (\ref{Du}) one has
$$c_0\leq HD(\tilde{\Lambda}_{i_0})/(\tilde{C}+1)\leq D_s(\tilde{\Lambda}_{i_0}) \ \mbox{and also}\ c_0\leq HD(\tilde{\Lambda}_{i_0})/(\tilde{C}+1)\leq D_u(\tilde{\Lambda}_{i_0})$$
then, Corollary \ref{uniform} applied to $\tilde{\Lambda}_{i_0}$ let us show the existence of two horseshoes $\Lambda ^s(t_0)$ and  $\Lambda ^u(t_0)$ of $\Lambda$ such that
$$D_u(\Lambda ^u(t_0))>D_u(\tilde{\Lambda}_{i_0})-\epsilon/2k, \quad D_s(\Lambda ^s(t_0))>D_s(\tilde{\Lambda}_{i_0})-\epsilon/2k,$$
$$\Lambda ^u(t_0)\cup \Lambda ^s(t_0) \subset \Lambda_{(t_0+\delta/2)-\delta}=\Lambda_{t_0-\delta/2},$$ 
and for every $x\in \Lambda ^u(t_0) \cup \Lambda ^s(t_0)$ the set $(\tilde{\Lambda}_{i_0})_{\epsilon/2k}(x)$ is neither bounded below nor bounded above.

Now, suppose there exists a subhorseshoe $\widetilde{\Lambda}\subset\Lambda_q$ for some $q<t_0$ with $\Lambda ^u(t_0) \cup \Lambda ^s(t_0) \subset \widetilde{\Lambda}$, then as $\Lambda_t\subset P$ for $t_0<t\leq t(h)$, we have by (\ref{a-L}) and Lemma \ref{L1}
\begin{eqnarray*}
L_{\varphi, f}(t_0)+a/2 &<& L_{\varphi, f}(t_0)+a-\epsilon/k <HD(\ell_{\varphi,f}(P))-\epsilon/k =HD(\tilde{\Lambda}_{i_0}) -\epsilon/k \\ &<& D_u(\Lambda ^u(t_0))+D_s(\Lambda ^s(t_0)) \leq HD(\widetilde{\Lambda})=HD(\ell_{\varphi,f}(\widetilde{\Lambda})) \leq HD(\ell_{\varphi,f}(\Lambda_q))\\ &\leq& \sup \limits_{s <t_0} HD(\ell_{\varphi,f}(\Lambda_s)) = L_{\varphi, f}(t_0)
\end{eqnarray*}
which is a contradiction. Then, by definition, $ \Lambda^s(t_0)$ and $ \Lambda^u(t_0)$ do not connect before $t_0$.

On the other hand, fix $x\in \Lambda ^s(t_0)$, $y\in \Lambda ^u(t_0)$ with kneading sequences $(x_n)_{n\in \mathbb{Z}}$, respectively $(y_n)_{n\in \mathbb{Z}}$. As the sets $(\tilde{\Lambda}_{i_0})_{\epsilon/2k}(x)$ and $(\tilde{\Lambda}_{i_0})_{\epsilon/2k}(y)$ are nonempty, we can find two words $\theta$ and $\tilde{\theta}$ in $C(\tilde{\Lambda}_{i_0},n(\epsilon/2k))$ that appear respectively in the sequences $(x_n)_{n\in \mathbb{Z}}$ and $(y_n)_{n\in \mathbb{Z}}$ as sub-words and also appear in the kneading sequence of two points $\tilde{x}_1, \tilde{y}_1\in \tilde{\Lambda}_{i_0}$, i.e., $\tilde{x}_1 \in R(\theta;0)$, and $\tilde{y}_1 \in R(\tilde{\theta};0)$, $(x_{N_1}, \dots x_{N_1+\abs{\theta}-1})=\theta$ and $(y_{-N_2-\abs{\tilde{\theta}}+1}, \dots y_{-N_2})=\tilde{\theta}$ for some $N_1,N_2>0$.

As $\tilde{\Lambda}_{i_0}$ is a horseshoe, we can find a point $z_1\in \tilde{\Lambda}_{i_0}$ with kneading sequence of the form
\begin{eqnarray*}
\Pi(z_1)=(\dots,z_{-2},z_{-1};\theta,z_{\abs{\theta}}, \dots ,z_{\abs{\theta}+r_1}, \tilde{\theta},z_{\abs{\theta}+r_1+\abs{\tilde{\theta}}+1}, \dots)
\end{eqnarray*}
for some $r_1>0$. Then consider the point $z\in \Lambda$ with kneading sequence 
\begin{eqnarray*}
\Pi(z)= (\dots,x_{-2},x_{-1}; x_0, \dots, x_{N_1-1},\theta, 
z_{\abs{\theta}}, \dots ,z_{\abs{\theta}+r_1}, \tilde{\theta}, y_{-N_2+1},y_{-N_2+2},y_{-N_2+3},\dots)
\end{eqnarray*}
note that, by construction $z\in W^u(\Lambda ^s(t_0))\cap W^s(\Lambda ^u(t_0))\cap \tilde{P}$ where 
$$\tilde{P}=M(C(\Lambda ^u(t_0) \cup \Lambda ^s(t_0) \cup \tilde{\Lambda}_{i_0},n(\epsilon/2k)))=\bigcap \limits_{m \in \mathbb{Z}} \varphi ^{-m}(\bigcup \limits_{\theta\in C(\Lambda ^u(t_0) \cup \Lambda ^s(t_0) \cup \tilde{\Lambda}_{i_0},n(\epsilon/2k))}  R(\theta;0)).$$  
Analogously, we can find $\tilde{z}\in W^u(\Lambda ^u(t_0))\cap W^s(\Lambda ^s(t_0))\cap \tilde{P}$. 
Moreover, as $\Lambda ^u(t_0) \cup \Lambda ^s(t_0) \cup \tilde{\Lambda}_{i_0} \subset \Lambda_{t_0+\delta/2}$, reasoning as we did for $P$, we have $\tilde{P}\subset \Lambda_{k\cdot\epsilon/2k+t_0+\delta/2}=\Lambda_{\epsilon/2+t_0+\delta/2}.$ That is,
$$\Lambda ^s(t_0) \cup \Lambda ^u(t_0) \cup \mathcal{O}(z) \cup \mathcal{O}(\tilde{z}) \subset \Lambda_{\epsilon/2+t_0+\delta/2}$$
and using Proposition \ref{connection11} we get that $\Lambda^s(t_0)$ and $ \Lambda^u(t_0)$ connect before $t_0+\epsilon$.

We summarize our conclusions in the following proposition
 \begin{proposition}\label{conection}
 Take $\varphi \in \mathcal{U}^*$ with $HD(\Lambda)<1$,\ $f\in \mathcal{P}_{\varphi,\Lambda}$ and some discontinuity $t_0$ of the map
 $$t \rightarrow L_{\varphi, f}(t)=HD(\mathcal{L}_{\varphi, f}\cap (-\infinity,t)).$$ 
 Then, given $\epsilon >0$ there are two subhorseshoes $\Lambda^s(t_0)$ and $\Lambda^u(t_0)$ and some $0<\eta<\epsilon$ such that
 \begin{itemize}
     \item $\Lambda^s(t_0)\cup \Lambda^u(t_0) \subset \Lambda_{t_0-{\eta}}$,
     \item $ \Lambda^s(t_0)$ and $ \Lambda^u(t_0)$ do not connect before $t_0$,
     \item $ \Lambda^s(t_0)$ and $ \Lambda^u(t_0)$ connect before $t_0+\epsilon$.
 \end{itemize} 
 \end{proposition}
\begin{remark}
As in Remark \ref{rr}, this result also holds when $HD(\Lambda)\geq 1$ and $t_0<\tilde{c}_{\varphi,f}$. Note that, in our context, by Corollary \ref{max}, $L_{\varphi, f}$ is discontinuous in $\tilde{c}_{\varphi,f}$ if and only if $L_{\varphi, f}(\tilde{c}_{\varphi,f})<1.$
\end{remark}
 \begin{figure}[ht]
\centering
\includegraphics[width=1.0 \textwidth]{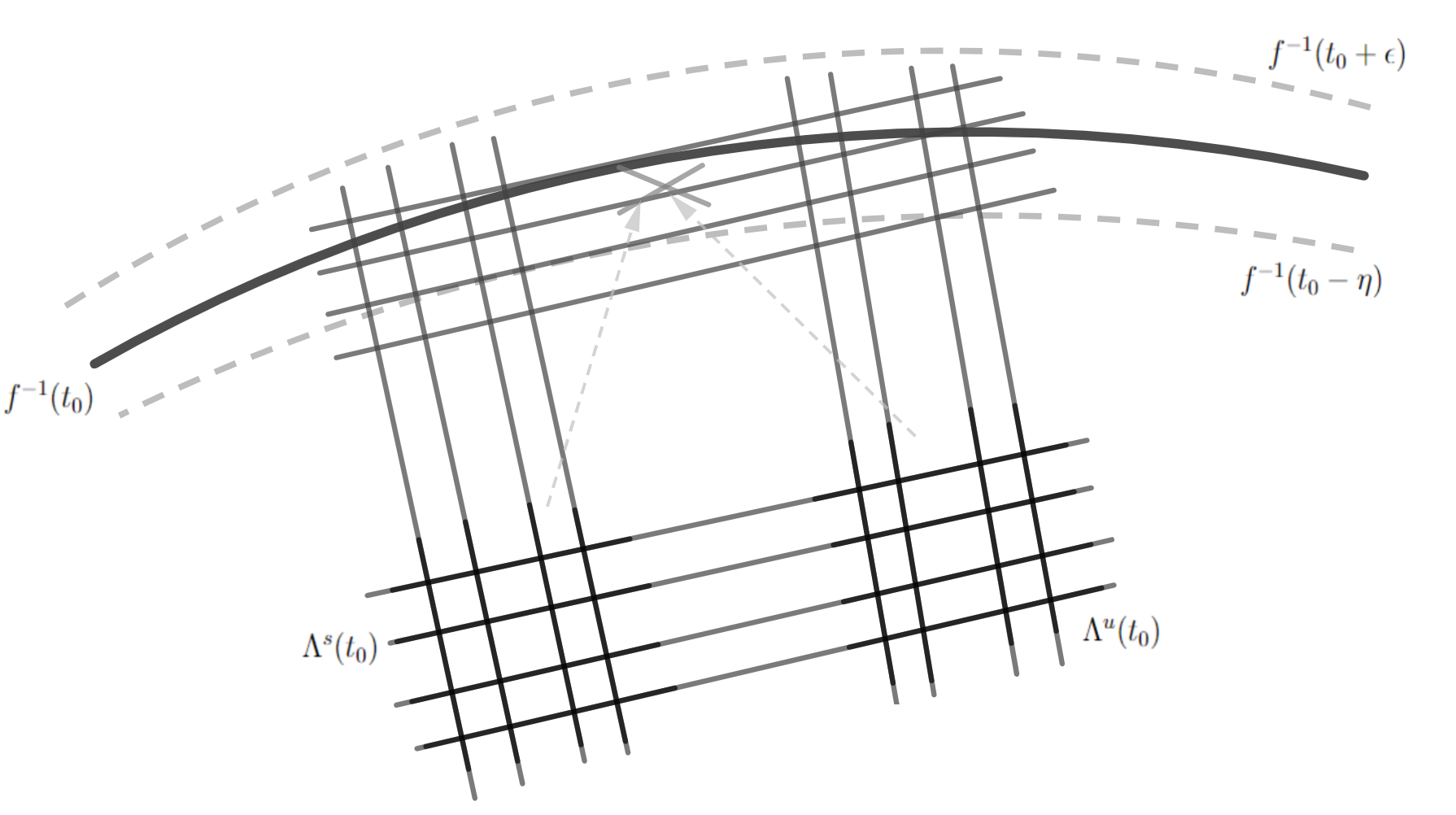}
\caption{The subhorseshoes $\Lambda^s(t_0)$ and $\Lambda^u(t_0)$ in Proposition \ref{conection}.}
\end{figure}

\subsection{Sequences of subhorseshoes}

In this subsection, we suppose existence of an infinite sequence of discontinuities of the map $L_{\varphi, f}$ in some closed subinterval of $I_{\varphi,f}$ that does not contain the first accumulation point of the Lagrange spectrum and then construct arbitrary large finite sequences of subhorseshoes with some specific properties. Observe that here is the first time when we use the hypothesis of the diffeomorphism being close to a conservative one.

Remember that any subhorseshoe $\tilde{\Lambda}_0$ of $\Lambda_0$ has a continuation $\tilde{\Lambda}\subset\Lambda$ for any $\varphi\in\mathcal{U}$. In \cite[Theorem $A$]{PV}, the authors showed that the maps $D_{\Lambda_0,u}:\mathcal{U}\rightarrow \mathbb{R}$ and $D_{\Lambda_0,s}:\mathcal{U}\rightarrow \mathbb{R}$ given by $D_{\Lambda_0,u}(\varphi)=D_u(\Lambda)$ and $D_{\Lambda_0,s}(\varphi)=D_u(\Lambda)$ are continuous and, in fact, the same proof also shows that the continuity of the maps $D_{\tilde{\Lambda}_0,u}(\varphi)=D_u(\tilde{\Lambda})$ and $D_{\tilde{\Lambda}_0,s}(\varphi)=D_u(\tilde{\Lambda})$ is uniform in the subhorseshoes. Moreover, as for $\varphi_0$ one can take $\tilde{C}=1$ in (\ref{conservative}) (see \cite[Remark 2.2]{CMM16}) then $D_u(\tilde{\Lambda}_0)=D_s(\tilde{\Lambda}_0)$ for any subhorseshoe $\tilde{\Lambda}_0$ of $\Lambda_0$ and, as a consequence, we can choose the neighborhood $ \mathcal{U}$  of $\varphi_0$ small enough such that for some constants $r_1,r_2$ with $r_1/r_2>999/1000$ and for any subhorseshoe $\tilde{\Lambda}$ of $\Lambda$ one has
\begin{equation}\label{c_1c_2}
  r_1D_s(\tilde{\Lambda})\leq D_u(\tilde{\Lambda})\leq r_2D_s(\tilde{\Lambda}).  
\end{equation}

 
Fix $\varphi \in \mathcal{U}^*$ with $HD(\Lambda)<1$, $f\in \mathcal{P}_{\varphi,\Lambda}$, some closed subinterval $I\subset I_{\varphi,f}$ that does not contain $c_{\varphi,f}$ and suppose we have an infinite sequence of discontinuities of $L_{\varphi, f}$ with $s\in I$ for every $s$ in the sequence. Then, as
$L_{\varphi,f}(\min I)\leq L_{\varphi, f}(s)\leq HD(\Lambda_{s})$, by (\ref{Ds}) and (\ref{Du})
\begin{equation}
    c\leq D_s(\Lambda_{s})\ \mbox{and}\ c\leq D_u(\Lambda_{s}),
\end{equation}
where $c=L_{\varphi,f}(\min I)/(\tilde{C}+1).$

Now, as the maps $t \mapsto D_u(\Lambda_{t})$ and $t \mapsto D_s(\Lambda_{t})$ are continuous (by Proposition \ref{R-generic}) and $D_u(\Lambda_{t})=D_s(\Lambda_{t})=0$ for $t<\min (f)$ and $D_u(\Lambda_{t})=D_u(\Lambda)$,  $D_s(\Lambda_{t})=D_s(\Lambda)$ for $t>\max (f)$. Then, they are uniformly continuous and so we can find some $\delta >0$ such that 
$$\abs{t-\bar{t}}<\delta\  \mbox{implies}\ \abs{D_u(\Lambda_{t})-D_u(\Lambda_{\bar{t}})} < 0.001c\  \mbox{and}\ \abs{D_s(\Lambda_{t})-D_s(\Lambda_{\bar{t}})}< 0.001c.$$
Also, for the sequence of discontinuities we have some accumulation point and unless pass to a sub-sequence, change the index set and discard some terms, we can suppose that $\{ t_n\}$ is of one of the next two types:
\begin{itemize}
    \item The sequence is strictly increasing $\{t_n\}_{n\geq 1}$ with $\lim \limits_{n\rightarrow \infinity}t_n:=t_0$ and $t_0-t_1<\delta$,
    \item The sequence is strictly increasing $\{t_n\}_{n\leq 0}$ with $\lim \limits_{n\rightarrow -\infinity}t_n:=t^*$ and $t_0-t^*<\delta.$
    \end{itemize}

In particular, for each $n$
\begin{equation}\label{estima0}
 0.999D_u(\Lambda_{t_0})=D_u(\Lambda_{t_0})- 0.001D_u(\Lambda_{t_0})\leq D_u(\Lambda_{t_0})- 0.001c<D_u(\Lambda_{t_n})  
\end{equation} 
and 
\begin{equation}\label{estima00}
 0.999D_s(\Lambda_{t_0})=D_s(\Lambda_{t_0})- 0.001D_s(\Lambda_{t_0})\leq D_s(\Lambda_{t_0})- 0.001c<D_s(\Lambda_{t_n}).  
\end{equation} 
Now, we will associate to each $n$ a pair of subhorseshoes of $\Lambda$. In fact, the two subhorseshoes $\Lambda^s(t_n)$ and $\Lambda^u(t_n)$ are given by Proposition \ref{conection} considering some $0<\epsilon_n< \min \{0.001,(t_{n+1}-t_n)/2\}$ and they satisfy 

\begin{itemize}
     \item $\Lambda^s(t_n)\cup \Lambda^u(t_n) \subset \Lambda_{t_n-{\eta_n}}$ for some $0<\eta_n<\epsilon_n$,
     \item $\Lambda^s(t_n)$ does not connect with $\Lambda^u(t_n)$ before $t_n$,
     \item $\Lambda^s(t_n)$ connects with $\Lambda^u(t_n)$ before $t_{n+1}$.
 \end{itemize} 
We are ready to prove the next proposition
  
  \begin{proposition}\label{cadeias}
  We can take $\theta \in \{s,u\}$ such that given $N\in \mathbb{N}$ arbitrary, there exists a sequence $n_1<n_2<...<n_N$ of elements of $\mathcal{I}$ (where $\mathcal{I}$ is the index set of the sequence $\{ t_n\}$) such that for $i,j\in  \{1,...,N \}$ with $i\neq j$,\ $\Lambda^{\theta}(t_{n_i})$ and $\Lambda^{\theta}(t_{n_j})$ do not connect before $\max \{ t_{n_i}, t_{n_j} \}.$
  \end{proposition}
 
 \begin{proof}
 
 We said that a sequence $n_1< n_2 <... <n_r$ of elements of $\mathcal{I}$ is a \emph{r-chain} if  $\Lambda^s(t_{n_i})$ connects with $\Lambda^s(t_{n_{i+1}})$ before $t_{n_{i+1}}$ for $i=1,\dots r-1$. Then we have two cases:
 \begin{itemize}
    \item There exists some $R\in \mathbb{N}$ such that there is no $r$-chain for $r>R$.
    \item There are $r$-chains with $r$ arbitrarily big.
    \end{itemize}

 We do the proof when the index set of the sequence is $\mathcal{I}=\{n\in \mathbb{Z}:n\geq 1\}$, and the other case follows similarly.  
 
 In the first case take a maximal $r_1$-chain beginning with 1; that is, a $r_1$-chain  $1=n_1<n_2<...<n_{r_1}$ such that for any $n>n_{r_1}$,\ $1=n_1<n_2<...<n_{r_1}<n$ is not a $(r_1+1)$-chain and then $\Lambda^s(t_{n_{r_1}})$ does not connect with $\Lambda^s(t_{n})$ before $t_n$.
 Next take a maximal $r_2$-chain beginning with $n_{r_1}+1$:\ $n_{r_1}+1=n^{(r_1)}_1<n^{(r_1)}_2<\dots < n^{(r_1)}_{r_2}$ then, as before, for $n^{(r_1)}_{r_2}< n$,\ $\Lambda^s(t_{n^{(r_1)}_{r_2}})$ does not connect with $\Lambda^s(t_{n})$ before $t_n$. Now 
 consider a maximal $r_3$-chain beginning with $n^{(r_1)}_{r_2}+1$: \ $n^{(r_1)}_{r_2}+1=n^{(r_1,r_2)}_1<n^{(r_1,r_2)}_2<\dots < n^{(r_1,r_2)}_{r_3}$ then for $n^{(r_1,r_2)}_{r_3}< n$, \ $\Lambda^s(t_{n^{(r_1,r_2)}_{r_3}})$ does not connect with $\Lambda^s(t_{n})$ before $t_n$. 
 
 Continuing in this way we can construct inductively an increasing sequence 
 $$\{ \tilde{n}_k \}_{k\geq 2}= \{ n_{r_k}^{(r_1,r_2,\dots,r_{k-1})} \}_{k\geq 2}$$  
 such that for $k_1, k_2 \geq 2$ with $k_1\neq k_2$, \ $\Lambda^{s}(t_{\tilde{n}_{k_1}})$ and $\Lambda^{s}(t_{\tilde{n}_{k_2}})$ does not connect before $\max \{ t_{\tilde{n}_{k_1}}, t_{\tilde{n}_{k_2}} \}$.
 
On the other hand, in the second case, take $r\in \mathbb{N}$ arbitrarily big and $n_1< n_2 <... <n_r$ some $r$-chain, then we affirm that for $i,j\in  \{1,...,r \}$ with $i\neq j$, \ $\Lambda^u(t_{n_i})$ and $\Lambda^u(t_{n_j})$ does not connect before $\max \{ t_{n_i}, t_{n_j} \}.$ In other case if for some $i_0,j_0\in  \{1,...,r \}$ with $i_0 < j_0$, $\Lambda^u(t_{n_{i_0}})$ and $\Lambda^u(t_{n_{j_0}})$ connect before $t_{n_{j_0}}$ then as by Corollary \ref{connection3}, \ $\Lambda^s(t_{n_{j_0}})$ connect with $\Lambda^s(t_{n_{i_0}})$ before $t_{n_{j_0}}$ and as also $\Lambda^s(t_{n_{i_0}})$ connects with $\Lambda^u(t_{n_{i_0}})$ before $t_{n_{i_0}+1}$ (and then before $t_{n_{j_0}}$). Applying two times more that corollary we have that $\Lambda^s(t_{n_{j_0}})$ connects with $\Lambda^u(t_{n_{j_0}})$ before $t_{n_{j_0}}$ that is a contradiction. From this follows the result.
 \end{proof}
 
Without loss of generality, we will suppose that in the previous proposition $\theta=u$ (for $\theta=s$ the argument is similar) and call $\Lambda^u(t_n)=\Lambda^n$. 

\subsection{Subhorseshoes and connection by periodic orbits}
In this subsection, we associate to every term of the sequence $\{\Lambda^n\}_{n\in\mathcal{I}}$ a periodic orbit with the property that if $\Lambda^n$ and $\Lambda^m$ are associated with the same periodic orbit then they connect before $\max \{t_n,t_m\}$.

In order to do that, given some $n$, remember the construction of $\Lambda^n$ given by Proposition \ref{conection}. A close inspection of the proof of that proposition shows that for some maximal invariant set, said $P^n$, that contains $\Lambda_{t_n}$ we took the subhorseshoe with maximal Hausdorff dimension $\Lambda_0^n\subset P^n$ and then applied Proposition \ref{uniform} in order to obtain the subhorseshoe $\Lambda^n$ with
\begin{equation}\label{estima1}
D_u(\Lambda^n)>(1-\epsilon_n/2k)D_u(\Lambda_0^n)>(1-\epsilon_n)D_u(\Lambda_0^n)>0.999D_u(\Lambda_0^n).
\end{equation}

Next, if $D_u(P^n)=D_u(\Lambda_2^n)$ where $\Lambda_2^n\subset P^n$ is a subhorseshoe of $\Lambda$, then as $\Lambda_0^n$ has maximal dimension, it follows that either $D_u(\Lambda_2^n)\leq D_u(\Lambda_0^n)$ or $D_s(\Lambda_2^n)\leq D_s(\Lambda_0^n)$. In the first case 
$$D_u(\Lambda_{t_n})\leq D_u(P^n)=D_u(\Lambda_2^n)\leq D_u(\Lambda_0^n)\leq \frac{r_2}{r_1}D_u(\Lambda_0^n)$$
and in the second, (\ref{c_1c_2}) let us conclude that $$ D_u(\Lambda_{t_n})\leq D_u(P^n)=D_u(\Lambda_2^n)\leq r_2D_s(\Lambda_2^n)\leq r_2D_s(\Lambda_0^n)\leq \frac{r_2}{r_1}D_u(\Lambda_0^n)$$ that is,
\begin{equation}\label{estima2}
  D_u(\Lambda_{t_n}) \leq \frac{r_2}{r_1}D_u(\Lambda_0^n).
\end{equation}

Now, by (\ref{du}), we can take $r_0$ big enough such that $2^{2023}<N_u(\Lambda_{t_0},r_0)$ and
\begin{equation}\label{estima3}
  \frac{\log N_u(\Lambda_{t_0},r_0)}{r_0-c_1}<1.001 D_u(\Lambda_{t_0}). 
\end{equation}
Set $\mathcal{B}_0=\mathcal{C}_u(\Lambda_{t_0},r_0)$, $N_0=N_u(\Lambda_{t_0},r_0)$ and for $n\in \mathcal{I}$, $M\in \mathbb{N}$ define the set
$$\mathcal{B}_M(\Lambda^n):=\{\beta=\beta_1\dots\beta_M :\forall \, 1\leq j\leq M, \,\, \beta_j\in\mathcal{B}_0 \,\,  \textrm{ and } \,\, \Pi^u(\Lambda^n)\cap I^u(\beta)\neq\emptyset\}.$$

Before continuing, we introduce some notation. Consider $\beta=\beta_{k_1}\beta_{k_2}...\beta_{k_ \ell}=a_1...a_p \in \mathcal{A}^{p}, \ \beta_{k_i}\in \mathcal{B}_0, \ 1\le i\le \ell$. We say that $n\in \{1,...,p\}$ is the n-th position of $\beta$. If $\beta_{k_i}\in \mathcal{A}^{n_{k_{i}}} $ we write $|\beta_{k_i}|=n_{k_i}$ for its length and $P(\beta_{k_i})=\{1,2,...,n_{k_i}\}$ for its set of positions as a word in the alphabet $\mathcal{A}$ and given $s \in P(\beta_{k_i})$ we call  $P(\beta,k_i;s)=n_{k_1}+...+n_{k_{i-1}}+s$ the position in $\beta$ of the position $s$ of $\beta_{k_i}$.	

Recall that the sizes of the intervals $I^u(\alpha)$ behave essentially submultiplicatively due the bounded distortion property of $\psi_u$ (equation (\ref{bdp1})) so that, one has   
$$|I^u(\beta)|\leq \exp(-M(r_0-c_1))$$
for any $\beta\in\mathcal{B}_M(\Lambda^n)$, and thus, $\{I^u(\beta):\beta\in \mathcal{B}_M(\Lambda^n)\}$ is a covering of $\Pi^u(\Lambda^n)$ by intervals of sizes $\leq \exp(-M(r_0-c_1))$. In particular for $M(\Lambda^n)=M_n$ big enough
\begin{eqnarray*}
\dfrac{\log\abs{\mathcal{B}_{M_n}(\Lambda^n)}}{\log N_0^{M_n}}&=&\dfrac{\dfrac{\log\abs{\mathcal{B}_{M_n}(\Lambda^n)}}{-\log \exp (-M_n (r_0-c_1))}}{\dfrac{M_n\cdot \log N_0}{M_n(r_0-c_1)}}\\ &\geq& \frac{\dfrac{\log\abs{\mathcal{B}_{M_n}(\Lambda^n)}}{-\log \exp (-M_n (r_0-c_1))}}{1.001 D_u(\Lambda_{t_0})} \quad (\textrm{by Equation (\ref{estima3})}) \\
\end{eqnarray*}
\begin{eqnarray*}
 \ \ \ \ \ \ \ \ \ \ \ \ \ &\geq& \frac{0.999D_u(\Lambda^n)}{1.001 D_u(\Lambda_{t_0})} \quad (\textrm{$M_n$ is big}) \\ &\geq&\ \dfrac{0.999\cdot0.999D_u(\Lambda_0^n)}{1.001 D_u(\Lambda_{t_0})}\quad (\textrm{by Equation (\ref{estima1})})\\&\geq& \dfrac{r_1}{r_2}\dfrac{0.999\cdot0.999D_u(\Lambda_{t_n})}{1.001 D_u(\Lambda_{t_0})}\quad (\textrm{by Equation (\ref{estima2})})\\ &\geq& \dfrac{r_1}{r_2}\dfrac{0.999\cdot0.999\cdot0.999}{1.001}\quad (\textrm{by Equation (\ref{estima0})})\\&>& 0.999^4/1.001 \\&>& 991/1000.
\end{eqnarray*}
Thus we have proved the next result 
\begin{lemma}\label{Btam}
Given $n\in \mathcal{I}$ and $M_n$ large 
$$ \abs{\mathcal{B}_{M_n}(\Lambda^n)}\geq N_0^{\frac{991}{1000}\cdot M_n}.$$
\end{lemma}

Remember that $f\in\mathcal{R}_{\varphi, \Lambda}$ where $\mathcal{R}_{\varphi, \Lambda}$ was defined in Section \ref{pre} above. Then, we can suppose, unless refining the initial Markov partition $\{R_a\}_{a\in\mathcal{A}}$, that the restriction of $f$ to each of the intervals $\{i_a^s\}\times I_a^u$, $a\in\mathcal{A}$, is strictly monotone and, furthermore, for some constant $c_4>0$, the following estimates hold 
\begin{eqnarray}\label{e.c1}
|f(\underline{\theta}^{(1)};a_1\dots a_n a_{n+1}\underline{\theta}^{(3)})-f(\underline{\theta}^{(1)};a_1\dots a_n a_{n+1}'\underline{\theta}^{(4)})| > c_4\cdot |I^u(a_1\dots a_n)| \\
|f(\underline{\theta}^{(1)}a_{m+1} a_m\dots ;a_1\underline{\theta}^{(3)})-f(\underline{\theta}^{(2)}a_{m+1}' a_m\dots ;a_1\underline{\theta}^{(3)})| > c_4\cdot |I^s(a_1\dots a_m)| \nonumber
\end{eqnarray}
whenever $a_{n+1}\neq a_{n+1}'$, $a_{m+1}\neq a_{m+1}'$ and $\underline{\theta}^{(1)}, \underline{\theta}^{(2)}\in\mathcal{A}^{\mathbb{Z}^-}$, $\underline{\theta}^{(3)}, \underline{\theta}^{(4)}\in\mathcal{A}^{\mathbb{N}}$ are admissible. 

Moreover, we observe that, since $f\in C^2$, there exists $c_5>0$ such that we also have the following estimates:
\begin{eqnarray}\label{e.c2}
|f(\underline{\theta}^{(1)};a_1\dots a_n a_{n+1}\underline{\theta}^{(3)})-f(\underline{\theta}^{(1)};a_1\dots a_n a_{n+1}'\underline{\theta}^{(4)})| < c_5 \cdot |I^u(a_1\dots a_n)| \\
|f(\underline{\theta}^{(1)}a_{m+1} a_m\dots ;a_1\underline{\theta}^{(3)})-f(\underline{\theta}^{(2)}a_{m+1}' a_m\dots ;a_1\underline{\theta}^{(3)})| < c_5 \cdot |I^s(a_1\dots a_m)| \nonumber
\end{eqnarray}   
whenever $a_{n+1}\neq a_{n+1}'$, $a_{m+1}\neq a_{m+1}'$ and $\underline{\theta}^{(1)}, \underline{\theta}^{(2)}\in\mathcal{A}^{\mathbb{Z}^-}$, $\underline{\theta}^{(3)}, \underline{\theta}^{(4)}\in\mathcal{A}^{\mathbb{N}}$ are admissible.

Next, we give a definition

\begin{definition}\label{good-position} Given $n\in \mathcal{I}$, $ M\in \mathbb{N}$ and $\beta=\beta_1\dots\beta_M \in \mathcal{B}_M(\Lambda^n) $ with $\beta_i\in\mathcal{B}_0$ for all $1\leq i\leq M$, we say that $j\in\{1,\dots,M\}$ is a \emph{M-right-good position} of $\beta$ if there are two elements of $\mathcal{B}_M(\Lambda^n)$
$$\beta^{(p)}=\beta_1\dots\beta_{j-1}\beta_j^{(p)}\dots\beta_M^{(p)}, \quad p=1, 2$$
with  $\beta_i^{(p)}\in\mathcal{B}_0$ for all $j\leq i\leq M, \ p=1, 2$ and  such that $\sup I^u(\beta_j^{(1)})<\inf I^u(\beta_j)< \sup I^u(\beta_j)<\inf I^u(\beta_j^{(2)})$, i.e., the interval $I^u(\beta_j)$ is located between $I^u(\beta_j^{(1)})$ and $I^u(\beta_j^{(2)})$. 

Similarly, we say that $j\in\{1,\dots, M\}$ is a \emph{M-left-good position} of $\beta$ if there are two elements of $\mathcal{B}_M(\Lambda^n)$
$$\beta^{(p)}=\beta_1^{(p)}\dots\beta_j^{(p)}\beta_{j+1}\dots\beta_M, \quad p=3, 4$$
with  $\beta_i^{(p)}\in\mathcal{B}_0$ for all $1\leq i\leq j, \ p=3, 4$ such that $\sup I^s((\beta_j^{(3)})^T)<\inf I^s(\beta_j^T)< \sup I^s(\beta_j^T)<\inf I^s((\beta_j^{(4)})^T)$, i.e., the interval $I^s(\beta_j^T)$ is located between $I^s((\beta_j^{(3)})^T)$ and $I^s((\beta_j^{(4)})^T)$.

Finally, we say that $j\in\{1,\dots, M\}$ is a \emph{M-good position} of $\beta$ if it is both a M-right-good and a M-left-good position of $\beta$.
\end{definition}
The bounded distortion property (Equation (\ref{bdp2})) let us fix $J\in \mathbb{N}$ big enough such that for $\beta_1\beta_2\dots\beta_J$ and $\beta_{J+1}\beta_{J+2}$ admissible with $\beta_1,\beta_2,\dots,\beta_J,\beta_{J+1},\beta_{J+2} \in\mathcal{B}_0=\mathcal{C}_u(\Lambda_{t_0},r_0)$ 
$$|I^u(\beta_1\beta_2\dots\beta_J)|\leq |I^s((\beta_{J+1}\beta_{J+2})^T)|$$
and 
$$|I^s((\beta_1\beta_2\dots\beta_J)^T)|\leq |I^u(\beta_{J+1}\beta_{J+2})|.$$

Set $k:=8J N_0^2$ (observe that $k$ does not depend on $n$). The next lemma says that most positions of some word of $\mathcal{B}_{5N_n k}(\Lambda^{n})$ are $5N_n k$-good.

\begin{lemma}\label{good} For $N_n$ big enough, there exists $\beta_n \in \mathcal{B}_{5N_n k}(\Lambda^{n})$ such that the number of  $5N_n k$-good positions of  $\beta_n$ is greater or equal than $49N_n k/10$.

\end{lemma}

\begin{proof}
Let us first estimate the cardinality of the subset of $\mathcal{B}_{5N_n k}(\Lambda^{n})$ consisting of words $\beta$ such that at least $N_n k/20$ positions are not  $5N_n k$-right-good: Once we fix a set of $m\geq N_n k/20$, $5N_n k$-right-bad (i.e., not  $5N_n k$-right-good) positions, if $j$ is a  $5N_n k$-right-bad position and $\beta_1, \dots, \beta_{j-1}\in\mathcal{B}_0$ were already chosen, then by definition, it follows that there are at most two options for $\beta_j\in\mathcal{B}_0$ which correspond to the leftmost and rightmost subintervals of $I^u(\beta_1\dots\beta_{j-1})$ of the form $I^u(\beta_1\dots\beta_{5N_n k})$ intersecting $\pi^u(\Lambda^n)$.

In particular, once a set of $m\geq N_n k/20$,    $5N_n k$-right-bad positions is fixed, the quantity of words in $\mathcal{B}_{5N_n k}(\Lambda_{n})$ with this set of $m$,  $5N_n k$-right-bad positions is less than or equal to  
$$2^m\cdot N_0^{5N_n k-m}\leq 2^{N_n k/20} \cdot N_0^{99N_n k/20}.$$
Therefore, the quantity of words in $\mathcal{B}_{5N_n k}(\Lambda^{n})$ with at least $N_n k/20$,  $5N_n k$-right-bad positions is less than or equal to
$$2^{5N_n k}\cdot 2^{N_n k/20}\cdot N_0^{99N_n k/20} = 2^{101N_n k/20}\cdot N_0^{99N_n k/20}.$$

Analogously, the quantity of words in $\mathcal{B}_{5N_n k}(\Lambda^{n})$ with at least $N_n k/20$,  $5N_n k$-left-bad positions is bounded by $2^{101N_n k/20}\cdot N_0^{99N_n k/20}$.

It follows that the set of words $\beta \in \mathcal{B}_{5N_n k}(\Lambda_{n})$ with at least $N_n k/10$, 
$5N_n k$-bad (i.e., not  $5N_n k$-good) positions has cardinality less or equal than $2.2^{101N_n k/20}\cdot N_0^{99N_n k/20}$.

Since $\abs{\mathcal{B}_{5N_n k}(\Lambda^{n})}>N_0^{991N_n k/200}$ (by Lemma \ref{Btam}) and $2^{1+101N_n k/20}\cdot N_0^{99N_n k/20}<N_0^{991N_n k/200}$ (from our choices of $r_0$, $N_0$ large), we have that there exists some $\beta_n \in \mathcal{B}_{5N_n k}(\Lambda^{n})$ with less than $N_n k/10$, $5N_n k$-bad positions. That is, with at least $5N_n k-N_n k/10=49N_n k/10$ good positions. 
\end{proof}

Given $n\in \mathcal{I}$ take $N_n$ big enough as in Lemma \ref{good} and such that for two elements $x,y\in \Lambda$ if their kneading sequences coincide in the central block (centered at the zero position) of size $2N_n+1$ then $\abs{f(x)-f(y)}<\eta_n/2$ where $\Lambda^s(t_n)\cup \Lambda^u(t_n) \subset \Lambda_{t_n-{\eta_n}}$.
 
The next proposition shows that the notion of good positions allows us to have some control over the values that $f$ takes in some rectangles. 

\begin{proposition}\label{control}
If $\beta_n=\beta^n_1\beta^n_2\dots\beta^n_{5N_nk}$ with $\beta^n_r\in \mathcal{B}_0$ for $i=1, \dots, 5N_nk$ is as in the previous lemma and for some $1<i<j<5N_nk$, the positions $i-1,i,j,j+1$ are $5N_nk$-good positions of $\beta_n$ and $j-i \geq J$. Then for each $i\leq s \leq j$ and $\bar{n} \in P(\beta^n_s)$ if $\eta=\beta^n_{i-1}\beta^n_{i}\dots\beta^n_j\beta^n_{j+1}$ and $x\in R(\eta;P(\eta,s;\bar{n}))\cap \Lambda$ we have $f(x)<t_n$.
\end{proposition}

\begin{proof}
The arguments are similar to those of \cite[Proposition 2.9]{CMM16}. Consider $\underline{\theta}^{(2)}\in\mathcal{A}^{\mathbb{N}}$ and $\underline{\theta}^{(1)}\in\mathcal{A}^{\mathbb{Z}^-}$ such that $\underline{\theta}^{(1)}\beta^n_{i-1};\beta^n_i\beta^n_{i+1}\dots\beta^n_{j-1}\beta^n_j\beta^n_{j+1}\underline{\theta}^{(2)}\in \Sigma_{\mathcal{B}}$. With this notation, our task is equivalent to show that
\begin{equation}\label{e.decreasing-f}f(\sigma^{\ell}(\underline{\theta}^{(1)}\beta^n_{i-1};\beta^n_i\beta^n_{i+1}\dots\beta^n_{j-1}\beta^n_j\beta^n_{j+1}\underline{\theta}^{(2)}))< t_n
\end{equation}
for all $0\leq \ell\leq m_1+m+m_2-1$ where $\beta^n_i=a_1\dots a_{m_1}$, $\beta^n_{i+1}\dots\beta^n_{j-1}=b_{1}\dots b_{m}$ and $\beta^n_j=d_{1}\dots d_{m_2}$.

First we deal with positions of the word $\beta^n_{i+1}\beta^n_{i+2}\dots\beta^n_j\beta^n_{j-1}$, that is, we consider $m_1\leq \ell\leq m_1+m-1.$ Write $\ell=m_1-1+r$ so that
\begin{equation}\label{e.sigma-l-B}\sigma^{\ell}(\underline{\theta}^{(1)}\beta^n_{i-1};\beta^n_i\beta^n_{i+1}\dots\beta^n_{j-1}\beta^n_j\beta^n_{j+1}\underline{\theta}^{(2)}) = \underline{\theta}^{(1)}\beta^n_{i-1}\beta^n_i b_1\dots b_{r-1}; b_r\dots b_{m}\beta^n_j\beta^n_{j+1}\underline{\theta}^{(2)}
\end{equation}
and also suppose that $|I^s((\beta^n_i b_1\dots b_{r-1})^T)|\leq |I^u(b_r\dots b_{m}\beta^n_j)|$ (the conclusion when $|I^u(b_r\dots b_{m}\beta^n_j)|<|I^s((\beta^n_i b_1\dots b_{r-1})^T)|$ follows similarly).

By definition of $5N_nk$-good position, we have 
$$\sup I^s((\beta_i')^T)<\inf I^s((\beta^n_i)^T)< \sup I^s((\beta^n_i)^T)<\inf I^s((\beta_i'')^T)$$
and
$$\sup I^u(\beta_j ')<\inf I^u(\beta^n_j)< \sup I^u(\beta^n_j)<\inf I^u(\beta_j ''),$$
for some words $\beta_i',\beta_i'',\beta_j',\beta_j''\in\mathcal{B}_0$ verifying 
$$I^u(\beta_i'\beta^n_{i+1}\dots\beta^n_{j-1}\beta^n_j\beta^n_{j+1})\cap \pi^u(\Lambda^n)\neq\emptyset, \quad I^u(\beta_i''\beta^n_{i+1}\dots\beta^n_{j-1}\beta^n_j\beta^n_{j+1})\cap \pi^u(\Lambda^n)\neq\emptyset,$$ 
$$I^u(\beta^n_{i-1}\beta^n_i\beta^n_{i+1}\dots\beta^n_{j-1}\beta_j')\cap \pi^u(\Lambda^n)\neq\emptyset, \quad I^u(\beta^n_{i-1}\beta^n_i\beta^n_{i+1}\dots\beta^n_{j-1}\beta_j'')\cap \pi^u(\Lambda^n)\neq\emptyset.$$

Choose $\beta_j^{\ast}\in\{\beta_j', \beta_j''\}$ such that 
$$f(\underline{\theta}^{(1)}\beta^n_{i-1}\beta^n_i b_1\dots b_{r-1}; b_r\dots b_{m}\beta^n_j\beta^n_{j+1}\underline{\theta}^{(2)})<f(\underline{\theta}^{(1)}\beta^n_{i-1}\beta^n_i b_1\dots b_{r-1}; b_r\dots b_{m}\beta_j^{\ast}\underline{\theta}^{(4)})$$
for any $\underline{\theta}^{(4)}\in\mathcal{A}^{\mathbb{N}}$. By (\ref{e.c1}), it follows that  
\begin{eqnarray*}
& &f(\underline{\theta}^{(1)}\beta^n_{i-1}\beta^n_i b_1\dots b_{r-1}; b_r\dots b_{m}\beta^n_j\beta^n_{j+1}\underline{\theta}^{(2)}) + c_4|I^u(b_r\dots b_{m}\beta^n_j)| \\ & &< f(\underline{\theta}^{(1)}\beta^n_{i-1}\beta^n_i b_1\dots b_{r-1}; b_r\dots b_{m}\beta_j^{\ast}\underline{\theta}^{(4)}).
\end{eqnarray*}
On the other hand, by (\ref{e.c2}), we also know that, for any $\underline{\theta}^{(3)}\in\mathcal{A}^{\mathbb{Z}^-}$
\begin{eqnarray*}
& & |f(\underline{\theta}^{(3)}\beta^n_{i-1}\beta^n_i b_1\dots b_{r-1}; b_r\dots b_{m}\beta_j^{\ast}\underline{\theta}^{(4)}) - f(\underline{\theta}^{(1)}\beta^n_{i-1}\beta^n_i b_1\dots b_{r-1}; b_r\dots b_{m}\beta_j^{\ast}\underline{\theta}^{(4)})| \\ & & < c_5 |I^s((\beta^n_{i-1}\beta^n_i b_1\dots b_{r-1})^T)|
\end{eqnarray*}
From these estimates, we obtain that 
$$f(\underline{\theta}^{(1)}\beta^n_{i-1}\beta^n_i b_1\dots b_{r-1}; b_r\dots b_{m}\beta^n_j\beta^n_{j+1}\underline{\theta}^{(2)}) + c_4|I^u(b_r\dots b_{m}\beta^n_j)|<$$
$$f(\underline{\theta}^{(3)}\beta^n_{i-1}\beta^n_i b_1\dots b_{r-1}; b_r\dots b_{m}\beta_j^{\ast}\underline{\theta}^{(4)}) + c_5 e^{c_1}|I^s((\beta^n_{i-1})^T)|\cdot |I^s((\beta^n_i b_1\dots b_{r-1})^T)|$$
for any $\underline{\theta}^{(3)}\in\mathcal{A}^{\mathbb{Z}^-}$ and $\underline{\theta}^{(4)}\in\mathcal{A}^{\mathbb{N}}$. 

Since we are supposing that $|I^s((\beta^n_i b_1\dots b_{r-1})^T)|\leq |I^u(b_r\dots b_{m}\beta^n_j)|$, we conclude
$$f(\underline{\theta}^{(1)}\beta^n_{i-1}\beta^n_i b_1\dots b_{r-1}; b_r\dots b_{m}\beta^n_j\beta^n_{j+1}\underline{\theta}^{(2)})<$$
$$f(\underline{\theta}^{(3)}\beta^n_{i-1}\beta^n_i b_1\dots b_{r-1}; b_r\dots b_{m}\beta_j^{\ast}\underline{\theta}^{(4)}) - (c_4- c_5 e^{c_1} |I^s((\beta^n_{i-1})^T)|)\cdot |I^u(b_r\dots b_{m}\beta^n_j)|.$$

Next, we note that if $r_0\in\mathbb{N}$ is sufficiently large, $c_5 e^{c_1}.|I^s((\beta^n_{i-1})^T)|<c_4/2$ . In particular, we have that 
\begin{eqnarray}\label{e.decreasing-f-caseIa}f(\underline{\theta}^{(1)}\beta^n_{i-1}\beta^n_i b_1\dots b_{r-1}; b_r\dots b_{m}\beta^n_j\beta^n_{j+1}\underline{\theta}^{(2)})< \\
f(\underline{\theta}^{(3)}\beta^n_{i-1}\beta^n_i b_1\dots b_{r-1}; b_r\dots b_{m}\beta_j^{\ast}\underline{\theta}^{(4)}) - (c_4/2)\cdot |I^u(b_r\dots b_{m}\beta^n_j)|\nonumber
\end{eqnarray}
for any $\underline{\theta}^{(3)}\in\mathcal{A}^{\mathbb{Z}^-}$ and $\underline{\theta}^{(4)}\in\mathcal{A}^{\mathbb{N}}$. 

Now, we recall that as $\beta_j^{\ast}\in\{\beta_j', \beta_j''\}$, one has 
$I^u(\beta^n_{i-1}\beta^n_i\beta^n_{i+1}\dots\beta^n_{j-1}\beta^*_j)\cap \pi^u(\Lambda^n)\neq\emptyset.$
By definition, this implies that there are $\underline{\theta}^{(3)}_*\in\mathcal{A}^{\mathbb{Z}^-}$ and $\underline{\theta}^{(4)}_*\in\mathcal{A}^{\mathbb{N}}$ with 
$$\underline{\theta}^{(3)}_*;\beta^n_{i-1}\beta^n_i\beta^n_{i+1}\dots\beta^n_{j-1}\beta^*_j\underline{\theta}^{(4)}_*\in\Sigma_{t_n},$$
and, in particular, by (\ref{e.sigma-l-B})
$$f(\sigma^{m_2+\ell}(\underline{\theta}^{(3)}_*;\beta^n_{i-1}\beta^n_ib_1\dots b_m\beta^*_j\underline{\theta}^{(4)}_*))=f(\underline{\theta}^{(3)}_*\beta^n_{i-1}\beta^n_ib_1\dots b_{r-1};b_{r} \dots b_m\beta^*_j\underline{\theta}^{(4)}_*))\leq t_n.$$
Combining this with (\ref{e.decreasing-f-caseIa}), we see that 
$$f(\underline{\theta}^{(1)}\beta^n_{i-1}\beta^n_i b_1\dots b_{r-1}; b_r\dots b_{m}\beta^n_j\beta^n_{j+1}\underline{\theta}^{(2)}) < t_n - (c_4/2)\cdot |I^u(b_r\dots b_{m}\beta^n_j)|$$
and then
\begin{equation}\label{e.delta1-definition}
f(\sigma^{\ell}(\underline{\theta}^{(1)}\beta^n_{i-1};\beta^n_i\beta^n_{i+1}\dots\beta^n_{j-1}\beta^n_j\beta^n_{j+1}\underline{\theta}^{(2)})) < t_n.
\end{equation}

Finally, the case when we deal with positions of the words $\beta^n_{i}$ or $\beta^n_{j}$ is similar with the previous one. We write
\begin{eqnarray*}\label{e.sigma-l-A}
& & \sigma^{\ell}(\underline{\theta}^{(1)}\beta^n_{i-1};\beta^n_i\beta^n_{i+1}\dots\beta^n_{j-1}\beta^n_j\beta^n_{j+1}\underline{\theta}^{(2)}) = \\ & & \underline{\theta}^{(1)}\beta^n_{i-1}a_1\dots a_{\ell};a_{\ell+1} \dots a_{m_1} \beta^n_{i+1}\dots \beta^n_{j-1}\beta^n_j\beta^n_{j+1}\underline{\theta}^{(2)} \nonumber
\end{eqnarray*}
for $0\leq\ell\leq\ m_1-1$, and 
\begin{eqnarray*}\label{e.sigma-l-D}
& & \sigma^{\ell}(\underline{\theta}^{(1)}\beta^n_{i-1};\beta^n_i\beta^n_{i+1}\dots\beta^n_{j-1}\beta^n_j\beta^n_{j+1}\underline{\theta}^{(2)}) = \\ & & \underline{\theta}^{(1)}\beta^n_{i-1}\beta^n_i\beta^n_{i+1}\dots\beta^n_{j-1}d_1 \dots d_{\ell - m_1-m};d_{\ell - m_1-m+1}\dots d_{m_2}\beta^n_{j+1}\underline{\theta}^{(2)} \nonumber
\end{eqnarray*}
for $m_1+m\leq\ell\leq m_1+m+m_2-1$.

Since $j-i\geq J$ and $\beta^n_{i-1},\beta^n_i,\dots,\beta^n_{j-1},\beta^n_{j}\in\mathcal{B}_0=\mathcal{C}_{u}(\Lambda_{t_0},r_0)$, it follows from our choice of $J$ that 
$$|I^u(a_{\ell+1} \dots a_{m_1} \beta^n_{i+1}\dots \beta^n_{j-1}\beta^n_j)|\leq |I^s((\beta^n_{i-1}a_1\dots a_{\ell})^T)|$$
for $0\leq\ell\leq m_1-1$, and 
$$|I^s((\beta^n_i\beta^n_{i+1}\dots\beta^n_{j-1}d_1 \dots d_{\ell - m_1-m})^T)|\leq |I^u(d_{\ell - m_1-m+1}\dots d_{m_2}\beta^n_{j+1})|$$
for  $m_1+m\leq\ell\leq m_1+m+m_2-1$. Arguing as before, one deduces that 
\begin{equation}\label{e.delta3-definition}
f(\sigma^{\ell}(\underline{\theta}^{(1)}\beta^n_{i-1};\beta^n_i\beta^n_{i+1}\dots\beta^n_{j-1}\beta^n_j\beta^n_{j+1}\underline{\theta}^{(2)}))<t_n-(c_4/2)\cdot|I^s((\beta^n_{i-1}a_1\dots a_{\ell})^T)|<t_n
\end{equation}
for $0\leq\ell\leq m_1-1$, and 
\begin{equation}\label{e.delta4-definition}
f(\sigma^{\ell}(\underline{\theta}^{(1)}\beta^n_{i-1};\beta^n_i\beta^n_{i+1}\dots\beta^n_{j-1}\beta^n_j\beta^n_{j+1}\underline{\theta}^{(2)}))<t_n-(c_4/2)\cdot|I^u(d_{\ell - m_1-m+1}\dots d_{m_2}\beta^n_{j+1})|<t_n
\end{equation}
for $m_1+m\leq\ell\leq m_1+m+m_2-1.$

In summary, from (\ref{e.delta1-definition}), (\ref{e.delta3-definition}), and (\ref{e.delta4-definition}) we deduce that (\ref{e.decreasing-f}) holds, as we wanted to see.
\end{proof}

 Consider $\beta_n=\beta^n_1\beta^n_2\dots\beta^n_{5N_n k}$ and divide its position set $I=\{1,2, \dots, 5N_n k \}$ in positions packages of size $N_n k$. In the central package  $I^*=\{2N_n k+1, 2N_n k+2, \dots, 3N_n k \}$, the number of $5N_n k$-bad positions is less than $5N_nk-49N_nk/10=N_n k/10$ and then subdividing that package now in $N_n$ package of positions of size $k$ we can find some package of size $k$ with less than $k/10$, $5N_n k$-bad positions, said 
 $$I^{**}=\{2N_nk+sk+1, 2N_n k+sk+2, \dots, 2N_n k+(s+1)k \} \ \mbox{for some} \ 0\leq s < N_n.$$

\begin{figure}[ht]
\centering
\includegraphics[width=1.0 \textwidth]{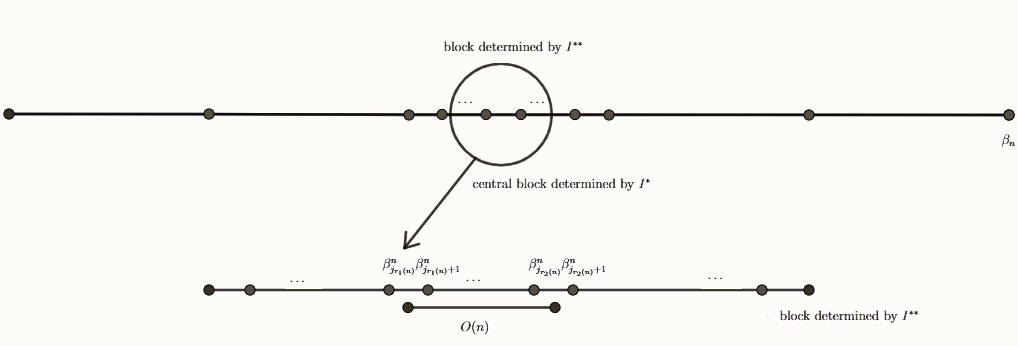}
\caption{Construction of $O(n)$.}
\end{figure}
 
Then we can find $\lceil2k/5\rceil$ positions 
$$2N_nk+sk+1\leq i_1\leq \dots \leq i_{\lceil2k/5\rceil}< 2N_n k+(s+1)k$$
such that $i_{r+1}\geq i_r+2$ for all $1\leq r<\lceil 2k/5\rceil$ and the positions $i_1, i_1+1, \dots, i_{\lceil2k/5\rceil},\\ i_{\lceil2k/5\rceil}+1$ are $5N_nk$-good. 

Since we took $k=8JN_0^2$, it makes sense to set  
$$j_r=i_{rJ} \quad \textrm{ for } 1\leq r\leq 3N_0^2$$
because $3JN_0^2<(16/5)JN_0^2= 2k/5$. In this way, we obtain positions such that 
$$j_{r+1}-j_r\geq 2J  \quad \textrm{ for } 1\leq r\leq 3N_0^2$$
and $j_1, j_1+1, \dots, j_{3N_0^2}, j_{3N_0^2}+1$ are $5N_nk$-good positions.

Since for $1\leq r\leq 3N_0^2$ the number of options for $(\beta^n_{j_r},\beta^n_{j_r+1})$ is at most $N_0^2$, we conclude that for some different $1\leq r_1(n),r_2(n)\leq 3N_0^2$ one has 
 $$(\beta^n_{j_{r_1(n)}},\beta^n_{j_{r_1(n)}+1})=(\beta^n_{j_{r_2(n)}},\beta^n_{j_{r_2(n)}+1})$$ 
then, we can define the following map:
\begin{eqnarray*}
   O: \mathcal{I} &\rightarrow& \bigcup\limits_{j=2}^{k-1}\mathcal{B}_0^j \\
   n &\rightarrow& \beta^n_{j_{r_1(n)}+1}\beta^n_{j_{r_1(n)}+2}\dots \beta^n_{j_{r_2(n)}}
\end{eqnarray*}

Next, we see that if for some $m,n\in \mathcal{I}$ we have $O(m)=O(n)$ then it is possible to go from $\Lambda^{m}$ to $\Lambda^{n}$ without leaving $\Lambda_{\max\{t_n,t_m\}}$ and staying arbitrarily close of the orbit of the periodic point $p=\Pi^{-1}(\overline{O(m)})$ for times arbitrarily big. More precisely, we have the following result
\begin{proposition}\label{m,n}
Take $m,n\in \mathcal{I}$ such that $O(m)=O(n)$. Then given $N\in \mathbb{N}$ and $\epsilon>0$ there exist some $x=x(N,\epsilon)\in W^u(\Lambda^m)\cap W^s(\Lambda^n)$ and $\overline{m}=\overline{m}(N,\epsilon)\in \mathbb{N}$ such that for $\overline{m} \leq i \leq \overline{m}+N$, $d(\mathcal{O}(p),\phi^i(x))<\epsilon$. Even more, we have $m_{\phi,f}(x) < \max\{t_n,t_m\}$.
\end{proposition}
\begin{remark}
 By symmetry, we also have the existence of some $y\in W^u(\Lambda^n)\cap W^s(\Lambda^m)$ and $\overline{n}\in \mathbb{N}$ with similar properties as $x$ and $\overline{m}$. 
\end{remark}
\begin{proof}
As $\beta_m \in \mathcal{B}_{5N_m k}(\Lambda^{m})$ and $\beta_n \in \mathcal{B}_{5N_n k}(\Lambda^{n})$ we can find $\theta_m^{1}, \theta_n^{1}\in\mathcal{A}^{\mathbb{Z}^-}$ and $\theta_m^{2}, \theta_n^{2}\in\mathcal{A}^{\mathbb{N}}$ such that 
$$\theta_m^{1};\beta_m\theta_m^{2}\in \Pi(\Lambda^{m})\quad \mbox{and} \quad \theta_n^{1};\beta_n\theta_n^{2}\in \Pi(\Lambda^{n}).$$
By Lemma \ref{good}, arguing as before; we can find positions $1\leq j_{r_0(m)}< N_m k$ and $1\leq j_{r_0(n)}< N_n k$ such that $j_{r_0(m)}$, $j_{r_0(m)}+1$ are $5N_m k$-good positions for $\beta_m$ and $j_{r_0(n)}$, $j_{r_0(n)}+1$ are $5N_n k$-good positions for $\beta_n$; and also positions $4N_m k+1\leq j_{r_3(m)}< 5N_m k$ and $4N_n k+1\leq j_{r_3(n)}< 5N_n k$ such that $j_{r_3(m)}$, $j_{r_3(m)}+1$ are $5N_m k$-good positions for $\beta_m$ and $j_{r_3(n)}$, $j_{r_3(n)}+1$ are $5N_n k$-good positions for $\beta_n$.

Define then for $R\in \mathbb{N}$
$$x_R=\theta_m^{1};\beta^m_1\beta^m_2\dots\beta^m_{j_{r_1(m)}}O(n)^R \beta^n_{j_{r_2(n)}+1}\beta^n_{j_{r_2(n)}+2}\dots\beta^n_{5N_n k}\theta_n^{2}.$$
Clearly, the proposition will be proved if we show that for some $t<\max\{t_n,t_m\}$, $x_R\in \Sigma_{t}$:

Let $l\in \mathbb{Z}$. In any of the next three cases:

\begin{itemize}
    \item If $\Pi^{-1}(\sigma^l(x_R))\in R(\eta;P(\eta,s;\bar{n}))$ for $\eta= \beta^n_{j_{r_1(n)}}\beta^n_{j_{r_1(n)}+1}\dots \beta^n_{j_{r_2(n)}} \beta^n_{j_{r_2(n)}+1} (=\\ \beta^m_{j_{r_1(m)}}\beta^m_{j_{r_1(m)}+1}\dots \beta^m_{j_{r_2(m)}} \beta^n_{j_{r_2(m)+1}})$, some $j_{r_1(n)}< s \leq j_{r_2(n)}$ and  $\bar{n} \in P(\beta^n_s)$.
   
   \item If $\Pi^{-1}(\sigma^l(x_R))\in R(\eta;P(\eta,s;\bar{n}))$ for $\eta=\beta^m_{j_{r_0(m)}}\beta^m_{j_{r_0(m)}+1}\dots \beta^m_{j_{r_1(m)}} \beta^m_{j_{r_1(m)+1}}$, some $j_{r_0(m)}< s \leq j_{r_1(m)}$ and  $\bar{n} \in P(\beta^m_s)$.
    
  \item If $\Pi^{-1}(\sigma^l(x_R))\in R(\eta;P(\eta,s;\bar{n}))$ for $\eta=\beta^2_{j_{r_2(n)}}\beta^2_{j_{r_2(2)}+1}\dots \beta^2_{j_{r_3(n)}} \beta^n_{j_{r_3(n)+1}}$, some $j_{r_2(n)}< s \leq j_{r_3(n)}$ and  $\bar{n} \in P(\beta^n_s)$

\end{itemize}
Proposition \ref{control} let us conclude that $f(\Pi^{-1}(\sigma^l(x_R)))< \max\{t_n,t_m\}$.
 
Let $r_1=\abs{\beta^m_1\beta^m_2\dots \beta^n_{j_{r_0(m)}}}$ then, for $l\leq r_1-1$
$$f(\Pi^{-1}(\sigma^l(x_R)))<f(\Pi^{-1}(\sigma^l(\theta_m^{1};\beta_m\theta_m^{2})))+\eta_m/2<t_m-\eta_m/2$$
because $\Lambda^{m}\subset \Lambda_{t_m-\eta_m}$ and as $j_{r_1(m)}-j_{r_0(m)}>2N_mk-N_mk=N_mk$ we have that $\sigma^l(x_R)$ coincides with $\sigma^l(\theta_m^{1};\beta_m\theta_m^{2})$ in the central block of size $2N_m+1$ centered at the zero position.  

Analogously, for $r_2=\abs{\beta^m_1\beta^m_2\dots\beta^m_{j_{r_1(m)}}O(n)^R \beta^n_{j_{r_2(n)}+1}\beta^n_{j_{r_2(n)}+2}\dots\beta^n_{j_{r_3(n)}}}$ , $j=r_2-\abs{\beta^n_1\beta^n_2\dots\beta^n_{j_{r_3(n)}}}$ and $l\geq r_2$
$$f(\Pi^{-1}(\sigma^l(x_R)))<f(\Pi^{-1}(\sigma^{l-j}(\theta_n^{1};\beta_n\theta_n^{2})))+\eta_n/2<t_n-\eta_n/2$$
because $\Lambda^{n}\subset \Lambda_{t_n-\eta_n}$ and as $j_{r_3(n)}-j_{r_2(n)}>4N_nk-3N_nk=N_nk$ we have that $\sigma^l(x_R)$ coincides with $\sigma^{l-j}(\theta_n^{1};\beta_n\theta_n^{2})$ in the central block of size $2N_n+1$ centered at the zero position.   

As the previous cases describe all the possibilities for $l\in \mathbb{Z}$ and for $l\leq r_1-1$ and $l\geq r_2$
we have uniform limitation for the values of $f(\Pi^{-1}(\sigma^l(x_R)))<\max\{t_n,t_m\}$ then we have proved the result.
\end{proof}

 Using Proposition \ref{m,n} we can prove that if for some $m,n\in \mathbb{N}$, $O(m)=O(n)$ then we can connect $\Lambda^m$ with $\Lambda^n$ without leaving $\Lambda_{max\{t_n,t_m\}}$ as is expressed in Definition \ref{conection of horseshoes1}
 
 \begin{corollary}\label{contradiction}
 Let $m,n\in \mathcal{I}$ such that $O(m)=O(n)$. Then $\Lambda^m$ connects with $\Lambda^n$ before $\max\{t_n,t_m\}$.
 \end{corollary}

 \begin{proof}
 Proposition \ref{m,n} let us find some $x,y \in \Lambda$ with $x\in W^u(\Lambda^m)\cap W^s(\Lambda^n)$, $y\in W^u(\Lambda^n)\cap W^s(\Lambda^m)$ and some $t<\max\{t_n,t_m\}$ \ such that 
$$\Lambda^n \cup \Lambda^m \cup \mathcal{O}(x) \cup \mathcal{O}(y) \subset \Lambda_t.$$ 
Then Proposition \ref{connection11} let us conclude that $\Lambda^n$ and $\Lambda^m$ connects before $\max\{t_n,t_m\}$. 
 \end{proof}

 \subsection{End of the proof of Theorem \ref{principal0} when the dimension is less than $1$}
We are ready to obtain the desired contradiction. As the map $O$ takes only a finite number of different values, said $M$.  Then by Corollary \ref{contradiction} it would be impossible to have a sequence $n_1<n_2<...<n_{M+1}$ of elements of $\mathcal{I}$ such that for $i,j\in  \{1,...,M+1 \}$ with $i\neq j$,\ $\Lambda^{n_i}$ and $\Lambda^{{n_j}}$ doesn't connect before $\max \{ t_{n_i}, t_{n_j}\}$ in contradiction with Proposition \ref{cadeias}.

\subsection{Proof of Theorem \ref{principal0} when the dimension is greater than or equal to $1$}

Consider $\varphi\in \mathcal{U}^*$ such that $HD(\Lambda)\geq 1$, $f\in \mathcal{P}_{\varphi,f}$ (see Proposition \ref{lagrange1}) and some closed subinterval $I\subset I_{\varphi,f}$ that doesn't contain neither $c_{\varphi,f}$ nor $\tilde{c}_{\varphi,f}$. Observe  that, in this case, by Corollary \ref{max}, $\max L_{\varphi, f}=1$ and then for $t<\tilde{c}_{\varphi,f}$ one has $L_{\varphi, f}(t)<1$. 

Take a hyperbolic set of finite type $P$ such that $$\Lambda_{\max I} \subset P\subset \Lambda_{\frac{\tilde{c}_{\varphi,f}+\max I}{2}}.$$
As before, the set $P$ admits a decomposition $P=\bigcup \limits_{i\in \mathcal{I}} \tilde{\Lambda}_i$ where $\mathcal{I}$ is a finite index set and for any $i\in \mathcal{I}$,\ $\tilde{\Lambda}_i$ is a subhorseshoe or a transient set. Note that if $i_0,i_1\in\mathcal{I}$ are different and $\tilde{\Lambda}_{i_0}$ and $\tilde{\Lambda}_{i_1}$ are subhorseshoes, then $\tilde{\Lambda}_{i_0}$ and $\tilde{\Lambda}_{i_1}$ don't connect before $\max I$.

Consider $s<\max I$, then we have 
$$\ell_{\varphi,f}(\Lambda_s)=\bigcup \limits_{i\in \mathcal{I}} \ell_{\varphi,f}(\tilde{\Lambda}_i\cap\Lambda_s)=\bigcup \limits_{\substack{i\in \mathcal{I}: \ \tilde{\Lambda}_i \ is\\ subhorseshoe }} \ell_{\varphi,f}(\tilde{\Lambda}_i\cap\Lambda_s)=\bigcup \limits_{\substack{i\in \mathcal{I}: \ \tilde{\Lambda}_i \ is\\ subhorseshoe }} \ell_{\varphi,f}((\tilde{\Lambda}_i)_s). $$
by taking union over $s<t$ where $t\leq \max I$, we conclude from this and Lemma \ref{L1} that
$$\mathcal{L}_{\varphi,f}\cap (-\infty,t)=\ell_{\varphi,f}(\Lambda)\cap (-\infty,t)=\bigcup \limits_{\substack{i\in \mathcal{I}: \ \tilde{\Lambda}_i \ is\\ subhorseshoe }} \ell_{\varphi,f}(\tilde{\Lambda}_i)\cap (-\infty,t)$$
and then, for $t\leq \max I$
$$L_{\varphi, f}(t)=\max\limits_{\substack{i\in \mathcal{I}: \ \tilde{\Lambda}_i \ is\\ horseshoe }} HD(\ell_{\varphi,f}(\tilde{\Lambda}_i)\cap (-\infty,t))=\max\limits_{\substack{i\in \mathcal{I}: \ \tilde{\Lambda}_i \ is\\ horseshoe }} L_i(t),$$
where $L_i(t)=HD(\ell_{\varphi,f}(\tilde{\Lambda}_i)\cap (-\infty,t))$ is associated to the horseshoe $\tilde{\Lambda}_i$ with
$$HD(\tilde{\Lambda}_i)=HD(\ell_{\varphi,f}(\tilde{\Lambda}_i))\leq HD(\ell_{\varphi,f}(\Lambda_{\frac{\tilde{c}_{\varphi,f}+\max I}{2}}))\leq L_{\varphi, f}\left(\frac{2\tilde{c}_{\varphi,f}+\max I}{3}\right)<1.$$
Observe that, as in Proposition \ref{R-generic}, the first part of the theorem also holds for subhorseshoes of $\Lambda$ with Hausdorff dimension less than $1$. Therefore, if we set $c_i=\min \{x: x\ \textrm{is an accumulation point of}\ \ell_{\varphi,f}(\tilde{\Lambda}_i)\}$, by Proposition \ref{first} there is some $i_0\in \mathcal{I}$ such that $c_{\varphi,f}=c_{i_0}$ and also by Corollary \ref{c} for any $i$ such that $c_{\varphi,f}< c_i$ the function $L_i$ doesn't contribute with any discontinuity close to $c_i$ to the discontinuity set of $L_{\varphi,f}$ (note that it is possible to have $c_i\geq \max I$ for some $i$). Then, we conclude that $L_{\varphi, f}$ has finitely many discontinuities in the interval $I$ as we wanted to see.


\begin{thebibliography}{10}

\bibitem{CMM16}
A.~Cerqueira, C. ~Matheus, C. G.~Moreira.
\newblock Continuity of Hausdorff dimension across generic dynamical Lagrange
  and Markov spectra.
\newblock {\em Journal of Modern Dynamics}, 12:151--174, 2018.
 
\bibitem{LMMR}
D.~Lima, C. Matheus, C. G. Moreira and S. Roma\~na 
\newblock   {\em Classical and Dynamical Markov and Lagrange spectra},
\newblock{\em World Scientific}, 2020.

\bibitem{LM}
D.~Lima and C. G. Moreira.
\newblock  Phase transtitions on the Markov and Lagrange dynamical spectra.
\newblock{\em Annales de L'Institute Henri Poincar\'e (C), Analyse non-lineaire}, 1--31,2020.

\bibitem{BK}
B. P. ~Kitchens.
\newblock {\em Symbolic Dynamics: One-sided, Two-sided and Countable State Markov Shifts},
\newblock Universitext, Springer, 1997.

\bibitem{MMan}
H. ~McCluskey and A. Manning.
\newblock Hausdorff dimension for horseshoes.
\newblock {\em Ergodic Theory and Dynamical Systems}, 3:251--260, 1983.

\bibitem{Shub} M. Shub. \emph{Global Stability of Dinamical Systems}. Springer-Verlag, 1986.

\bibitem{M50}
C. G. Moreira.
\newblock Geometric properties of images of cartesian products of regular Cantor
sets by differentiable real maps,
\newblock Preprint (2016) available at arXiv:1611.00933

\bibitem{M79}
A.~Markoff.
\newblock Sur les formes quadratiques binaires ind\'efinies.
\newblock {\em Math.Ann.}, 15:381--406, 1879.

\bibitem{GCD}
C. G. Moreira, C. Villamil and D. Lima.
\newblock Continuity of fractal dimensions in conservative generic Markov and Lagrange dynamical spectra,
\newblock Preprint (2023) available at arXiv:2305.07819

\bibitem{GC}
C. G. Moreira and C. Villamil.
\newblock Concentration of dimension in extremal points of left-half lines in the Lagrange spectrum,
\newblock Preprint (2023) available at arXiv:2309.14646

\bibitem{M3}
C.~G. Moreira.
\newblock Geometric properties of Markov and Lagrange spectra.
\newblock {\em Annals of Math.}, 188: 145--170, 2018.

\bibitem{MR2}
C.~G. Moreira and S.~Roma\~na.
\newblock On the Lagrange and Markov dynamical spectra.
\emph{Ergodic Theory and Dynamical Systems}, Volume 37, Issue 5, August 2017, pp. 1570 - 1591

\bibitem{CF}
T.~Cusick and M.~Flahive, 
\newblock{The Markoff and Lagrange spectra}, 
\emph{Mathematical Surveys and Monographs}, \textbf{30}. American Mathematical Society, Providence, RI, 1989. x+97 pp.

\bibitem{MY-10}
C.~G. Moreira and J.-C. Yoccoz.
\newblock Tangences homoclines stables pour des ensembles hyperboliques de
  grande dimension fractale.
\newblock {\em Annales Scientifiques de l'\'Ecole Normale Sup\'erieure},
  43:1--68, 2010.
  
\bibitem{PV}
J. Palis, J. and M. Viana.
\newblock On the continuity of Hausdorff dimension and limit capacity for horseshoes.
\newblock {\em Dynamical Systems Valparaiso 1986. Lecture Notes in Mathematics}, vol 1331. Springer 1988, Berlin, Heidelberg.

\bibitem{PT93}
J.~Palis and F.~Takens.
\newblock {\em Hyperbolicity and Sensitive chaotic dynamics at homoclinic
  biifurcations: fractal dimensios and infinitely many attractors}.
\newblock Cambridge Univ. Press, 1993.




\end{thebibliography}
\end{document}